\documentclass[arxiv,reqno,twoside,a4paper,11pt]{amsart}

\usepackage{mltools}
\usepackage{math62,mlthm10-REAR} 

\usepackage{xy} \xyoption{all} 
\usepackage{upref,mathrsfs,enumerate} 

\usepackage[symbol]{footmisc}
\usepackage{perpage}
\MakePerPage[2]{footnote} 

\usepackage[colorlinks=true,linkcolor=blue,citecolor=blue,backref=page]{hyperref}
\usepackage{REARmac}

\usepackage[scaled]{helvet} 
\usepackage{courier} 
\usepackage[mathbf]{euler}

\renewcommand{\smallsechook}{\relax}

\renewcommand{\MLprooffont}{\scshape}

\numberwithin{equation}{section}

\makeatletter
\let\appendixname\@empty
\makeatother
\begin{document}
\title[Divided differences, rearrangement Lemma, expansional formula]{
Divided differences in noncommutative geometry:
Rearrangement Lemma, functional calculus and expansional formula}

\author{Matthias Lesch}
\address{Mathematisches Institut,
Universit\"at Bonn,
Endenicher Allee 60,
53115 Bonn,
Germany}

\email{ml@matthiaslesch.de, lesch@math.uni-bonn.de}
\urladdr{www.matthiaslesch.de, www.math.uni-bonn.de/people/lesch}
\thanks{Partially supported by the Hausdorff Center for Mathematics}

\subjclass[2010]{Primary 46L87, 58B34; Secondary 47A60, 65D05}

%
\keywords{Rearrangement Lemma, divided difference, expansional formula,
modular curvature, topological tensor product, Banach algebra, functional
calculus, spectral measure, asymptotic expansion, noncommutative global
analysis} 
\copyrightinfo{2014-2015}{Matthias Lesch}

\begin{abstract}
We state a generalization of the Connes-Tretkoff-Moscovici
Rearrangement Lemma and give a surprisingly simple (almost trivial)
proof of it. Secondly, we put on a firm ground the multivariable
functional calculus used implicitly in the Rearrangement Lemma and
elsewhere in the recent modular curvature paper by Connes and
Moscovici \cite{ConMos2011}.  Furthermore, we show that the fantastic
formulas connecting the one and two variable modular functions of
loc.~cit. are just examples of the plenty recursion formulas which can
be derived from the calculus of divided differences. We show that the
functions derived from the main integral occurring in the
Rearrangement Lemma can be expressed in terms of divided differences
of the Logarithm, generalizing the ``modified Logarithm'' of
Connes-Tretkoff \cite{ConTre2011}.

Finally, we show that several expansion formulas related to the
Magnus expansion \cite{Mag1954} have a conceptual explanation in terms of a
multivariable functional calculus applied to divided differences.
\end{abstract}

\maketitle

\section{Introduction}  
\label{SecIntro}

This paper is inspired by the recent work on the spectral geometry
of non-commutative tori \cite{ConTre2011}, \cite{ConMos2011},
\cite{FarKha2012}, \cite{FarKha2013}. 

The striking novelty of the paper \cite{ConMos2011} is the occurrence
of universal one and two variable functions $K_0(s), H_0(s,t)$ in the
expression for the second heat coefficient \cite[(1)]{ConMos2011}\footnote{%
$\varphi_0$ is the natural trace on the non-commutative torus,
$\bigtriangleup$ is the flat Laplacian, $\bigtriangleup_\varphi$ the
conformal Laplacian with respect to the non-tracial weight
$\varphi(a)=\varphi_0(a e^{-h})$, $\Delta=e^{-h}\cdot e^{h}$ denotes
the modular operator, $\nabla=\log\Delta$.  $\square_\Re(h)$ is the
Dirichlet quadratic form in $\delta_1 h, \delta_2 h$, where 
$\delta_1, \delta_2$ are the natural derivations associated to the
$\R^2$--action on the non-commutative torus.  Finally, $\nabla_j$
signifies that $\nabla_j$ acts on the $j$-th factor. $e^h=:k^2$. 
}
\begin{equation}\label{EQIntro0}
a_2(a,\bigtriangleup_\varphi)=
\operatorname{Const}\cdot \varphi_0\Bl 
    a \bl K_0(\nabla)(\Delta h) + \frac 12
    H_0(\nabla_1,\nabla_2)(\square_\Re(h)) \br \Br.
\end{equation}
A rather fantastic aspect is that these functions satisfy relations of
the kind \cite[(4)]{ConMos2011}\footnote{$H$ and $K$ are modifications
of $H_0$ and $K_0$, for details see loc.~cit.}
\begin{equation}\label{EQIntro4}
H(a,b)= \frac{K(b)-K(a)}{a+b}
+\frac{K(a+b)-K(b)}{a}
-\frac{K(a+b)-K(a)}{b}.
\end{equation}
\Eqref{EQIntro4} is proved in loc.~cit.~by an a priori argument
\cite[Sec.~4.3]{ConMos2011}. 

On the other hand \Eqref{EQIntro4} is a sum of \emph{divided difference}s.
Namely, noting that $K$ in loc.~cit.~is an even function, we can
rewrite \Eqref{EQIntro4} as 
\begin{equation}\label{EQIntro5}
H(a,b)=[-a,b]K+[a+b,b]K-[a+b,a]K.
\end{equation}

So it seems as if divided differences could be the key to a lot of the
somewhat magic formulas occurring in this business.  And indeed when
thinking about \cite{ConMos2011} for a while the author noticed that
one stumbles over divided differences everywhere, notably in the
noncommutative Taylor expansion formula of the exponential function
and in concrete functions related to the Rearrangement Lemma, and soon
it became obvious that the role of divided differences in the subject
needs to be clarified. 

Divided differences are a standard tool in numerical analysis and they
can be calculated quite efficiently. We will recall the main facts
about them in Appendix \ref{SecDD} below.  To the best of our
knowledge their appearance in operator theory and functional calculus
is new.  In a different context, however, it was also observed in
\cite{BaxBru2011} that the Magnus expansion formula can be interpreted
in terms of the Genocchi-Hermite formula and hence related to divided
differences.

We now describe the content of the paper in more detail.

\subsection{Rearrangement Lemma and multivariable functional calculus}
An important technical tool for the calculation of heat coefficients
in the noncommutative setting is the Rearrangement Lemma which
informally reads
\begin{multline}\label{EQIntro1}
\int_0^\infty f_0(uk^2)\cdot b_1\cdot
f_1(uk^2)\cdot b_2\cldots b_p \cdot f_p(uk^2) \dint u\\
   = k^{-2} F(\modul 1;,\modul 1;\modul 2;,\ldots,
      \modul 1;\cldots\modul p;)(b_1\cldots b_p),
\end{multline}
where the function $F(s_1,\ldots,s_p)$ is
\[
  F(s) = \int_0^\infty f_0(u)\cdot f_1(us_1)\cldots f_p(u s_p) \dint u
\]
and $\modul j;$ signifies that the modular operator 
$\Delta = k^{-2}\cdot k^2$ acts on the $j$-th factor.
In \cite{ConMos2011} it is proved for the concrete integral
\begin{equation}\label{EQIntro2}
\int_0^\infty (uk^2)^{|\ga|+p-1}(1+u k^2)^{-\ga_0-1} \cdot b_1
\cdot (1+u k^2)^{-\ga_1-1}\cldots b_p\cdot
(1+u k^2)^{-\ga_p-1} \dint u,
\end{equation}
and the function
\begin{equation}\label{EQIntro3}
\fH(s,m)
:= \int_0^\infty
x^{|\ga|+p-1-m}\cdot (1+x)^{-\ga_0-1}\cdot
\prod_{j=1}^p (1+s_j x)^{-\ga_j-1} \dint x.
\end{equation}
The Rearrangement Lemma and the function $\fH(s,m)$ are crucial for
identifying the ingredients of \Eqref{EQIntro0} from a combinatorially
challenging expression for the resolvent expansion. The one and two
variable functions mentioned above are, after a change of variables,
simple linear combinations of basic $\fH(s,m)$ for a few values of
$\ga$.

The proof of \Eqref{EQIntro3} in loc.~cit.~consists of an intimidating
calculation involving explicit Fourier transforms of the factors of
the integrand after a change of variables. Since the Lemma has
appeared in several
versions of increasing complexity in the literature,
\cite[Lemma~6.2]{ConTre2011}, \cite[Lemma~6.2]{ConMos2011},
\cite[Lemma~4.2]{FarKha2012}, \cite[Prop.~3.4]{BhuMar2012},
\cite[Lemma~4.1]{FarKha2013}, we think that a systematic treatment
might be in order, also in light of possible generalizations
of the aforementioned papers to other noncommutative spaces.

One of the purposes of this note is to give a new proof of a fairly
general version of this Lemma. Our proof is not at all shorter than
the one in \cite[Lemma~6.2]{ConMos2011} but, at least the author
believes so, conceptually much simpler. We do not need explicit
Fourier transforms, all we use is the Spectral Theorem and the trivial
substitution $\int_0^\infty f(u\gl)\dint u =\gl\ii\int_0^\infty f(u)
\dint u$.
Namely, the Rearrangement Lemma is concerned with an integral,
\begin{equation}
\int_0^\infty f(uR_0, uR_1,\ldots, uR_n) \dint u,
\end{equation}
where $R_0,\ldots,R_n$ are commuting selfadjoint operators, and it
ultimately boils down to the justification of the ``operator
substitution'' $\tilde u = u R_0, du = R_0\ii d\tilde u$.

Secondly, we would like to put on a firm ground the functional
calculus which is implicitly used by the statement ``$\modul j;$
signifies that $\Delta$ acts on the $j$-th factor''. The authors hopes
that the current modest considerations will serve the community as he
has even heard the statement ``that these formulas should be
considered as formal since they are not based on a valid functional
calculus''. We will see that one should not be that pessimistic and
that the proper way to make sense of the notation $F(\modul
1;,\ldots)$ is the theory of tensor products of Banach and
$C^*$--algebras and the functional calculus for several commuting
operators. At the heart of the problem is the multiplication map
\begin{equation}
\mu_n: \sAn\ni a_0\oldots a_n\mapsto a_0\cldots a_n
\end{equation}
and the problem of extending it in a proper way to tensor product
completions. More concretely, $\mu_n$ extends by continuity to the
projective Banach algebra tensor product $\sAn_\gamma$.  On the other
hand a nice functional calculus for commuting selfadjoint operators is
available in the maximal $C^*$-tensor product $\sAn_\pi$. We suspect,
however, that $\mu_n$ does not extend by continuity to $\sAn_\pi$. We
circumvent this problem by establishing, for a selfadjoint element
$a\in\sA$, a smooth functional calculus in $\sAn_\gamma$ for the
commuting elements $a\pup j;=1_\sA\otimes\dots \otimes a\otimes \dots$
($a$ in slot $j$ counted from $0$).

\subsection{Divided differences} 
Coming back to divided differences and the formulas \Eqref{EQIntro4}
and \eqref{EQIntro5} we will show below that when dealing with the
integrand of \Eqref{EQIntro3}, divided difference occur in abundance
and the calculus of divided differences leads to a more or less
endless list of variations of \Eqref{EQIntro5}.

More concretely, we will express the function \Eqref{EQIntro3}
explicitly in terms of divided differences of the Logarithm:
\begin{equation}
\fH (s,m) = (-1)^{m+|\ga|+p-1}\cdot [1^{\ga_0+1}, s_1^{\ga_1+1}, \ldots,
s_p^{\ga_p+1}] \id^m\log.
\end{equation}

The modified logarithm $\cL_m$ of \cite[Lemma 3.2]{ConTre2011} 
is nothing but the divided difference 
\begin{equation}
\cL_m(s)= (-1)^m\cdot [1^{m+1},s]\log =(-1)^m\cdot [1^m,s]\cL_0,
\quad \cL_0(s)=\frac{\log(s)}{s-1},
\end{equation}
where $[1^m,s]f$ is an abbreviation for the divided difference
$[1,\ldots, 1,s]f$ with $m$ repetitions of $1$, \cf Secs.~\ref{ssEx2}
and \ref{ssA2}.  Note that $\cL_0(e^x)$ is the generating function for
the Bernoulli numbers, which occurs prominently in 
\cite{ConMos2011}\footnote{To be precise with 
$f(x)=\cL_0(e^x)$ we have 
\begin{equation*}
\frac 18 K(s) = \sum_{n=1}^\infty 
\frac{B_{2n}}{(2n)!} s^{2n-2}
     = \frac 1s (f(s) - f(0) - f'(0) s)
     = [0,s] f - [0,0] f = s [0,0,s] f.
\end{equation*}
}.

\subsection{Noncommutative Taylor expansion of the exponential function}
We show that the expansion formula for noncommutative variables $a$
and $b$ (\textit{cf., e.g.}, \cite[Sec.~6.1]{ConMos2011})
\begin{equation}\label{EQIntroExpansion}
e^{a+b} = e^a + \sum_{n=1}^\infty
     \int_{0\le s_n\le \ldots \le s_1\le 1} 
       e^{(1-s_1)a}\cdot b\cdot e^{(s_1-s_2)a}\cdot b\cldots b\cdot
       e^{s_na} \dint s
\end{equation}
can be interpreted nicely as an operator valued version of Newton's
interpolation formula involving divided differences
\begin{equation}
e^{a+b} = \sum_{n=0}^\infty \Bl[a\pup 0;,\ldots,a\pup n;] 
\exp_\gamma \Br(b\cldots b).
\end{equation}
This immediately leads to the following generalization
\begin{equation}
f(a+b) \sim_{b\to 0} \sum_{n=0}^\infty 
\bl[a\pup 0;,\ldots,a\pup n;]  f_\gamma \br(b\cldots b),
\end{equation}
for selfadjoint elements in a $C^*$--algebra and a Schwartz function
$f$.  The linear term in this expansion formula is at the heart of the
relations \Eqref{EQIntro4}, \eqref{EQIntro5}.  As an application we
give a conceptually new proof of the corresponding results in
\cite[Lemma 4.11 and Lemma 4.12]{ConMos2011}.

\subsection{Explicit examples} 
Finally, in Sec.~\ref{SecExamples} we discuss explicit examples of one
and two variable functions derived from \Eqref{EQIntro5} and compare
them to the explicit formulas given at the end of \cite{ConMos2011}.
In the preparation of Sec.~\ref{SecExamples} we used the open source
computer algebra system \texttt{Maxima}. However, the results as they
stand can be checked (a posteriori) by hand.

\subsection{} This paper is a byproduct of a recent joint project with
Henri Moscovici \cite{LesMos2015}; it is used in some of the
concrete calculations in Sec. 4 of that paper.

\subsection{Acknowledgment} I thank Alan Carey, Joachim Cuntz, Henri
Moscovici Markus Pflaum and Adam Rennie for helpful conversations and
suggestions.

\tableofcontents
\section{An abstract operator substitution Lemma}

\label{SecOSL}

\subsection{Notation}\label{ssNot}
$\N=\{0,1,2,\ldots\}, \Z, \R, \C$ denotes the natural numbers,
integers, real and complex numbers resp.  $\R_{\ge 0}$ denotes
$\bigsetdef{x\in\R}{x\ge 0}$, $\R_{>0}$, $\R_{<0}, \Z_{>0}, 
\Z_{\ge 0}$ etc. is used accordingly.  Instead of the clumsy 
$(\R_{\ge 0})^n$ we write $\R_{\ge 0}^n$.

We will frequently use the multiindex notation for partial derivatives
and factorials. Recall that if $\ga=(\ga_0,\ldots,\ga_n)\in \N^{n+1}$ 
is a multiindex then one abbreviates $\ga!:=\prod_j \ga_j!$,
$|\ga|:=\sum \ga_j$, and 
$\pl_x^\ga=\prod_j\pl_{x_j}^{\ga_j}, x=(x_0,\ldots x_n)$. 
Furthermore, we use the Pochhammer symbol for the rising and falling
factorial powers, see \Eqref{EQ3.1a}, \eqref{EQ3.1b}.

\subsection{} \label{ss2.1}
\newcommand\CastR{C^*(I, R_0,\ldots,R_n)}
Let $\sH$ be a Hilbert space and let $R_0,\ldots, R_n, n\ge 1,$ be
commuting positive selfadjoint operators in $\sH$, \ie~all operators
$R_j$ are assumed to be $\ge 0$ and invertible. These operators
generate a commutative unital $C^*$--subalgebra, $\sA=\CastR$, of the
$C^*$--algebra of bounded linear operators, $\sL(\sH)$, on the Hilbert
space $\sH$. By the Gelfand Representation Theorem, there exists a
compact subset $X\subset \prod_{j=0}^n \spec(R_j)\subset \C^{n+1}$ and
a $*$--isomorphism
\[
\Phi: C(X) \longrightarrow \CastR\subset \sL(\sH)
\]
which sends the constant function $1$ to the identity operator $I$ and
the function $x\mapsto x_j$ onto the operator $R_j$, $j=0, \ldots, n$.
$\Phi$ is called the spectral measure of $R_0,\ldots, R_n$, \cf
\cite[Chap.~12]{Rud91}. For a continuous function $f\in C(X)$ on
writes suggestively $f(R_0,\ldots, R_n):=\Phi(f)$. $\Phi$ may also be
viewed as an operator valued measure, \cf \cite[12.17]{Rud91}.  We
write $dE$ for the associated resolution of the identity in the sense
of loc.~cit. Then $f(R_0,\ldots, R_n)=\int_X f(\gl) \dint E(\gl)$.

For each pair of Hilbert space vectors $x,y\in \sH$ the spectral
measure $\Phi$ induces a complex Radon measure $E_{x,y}$ on $X$
by the identity
\[
\binn{ f(R_0,\ldots,R_n) x,y} = \int_X f(\gl)\dint E_{x,y}(\gl),
\quad f\in C(X).
\]

\begin{lemma}\label{LSpectralFubini} With the previously introduced notation let
$f:\R_{\ge 0}\times X\to \C$ be a continuous function satisfying the
integrability condition
\begin{equation}\label{EQIntegrability}
   \int_0^\infty \sup_{\gl\in X} |f(u,\gl)| \dint u <\infty.
\end{equation}   
Define $F:X\to\C$ by the parameter integral
\begin{equation*}
   F(\gl):= \int_0^\infty f(u,\gl) \dint u.
\end{equation*}
Then the integral $\int_0^\infty f(u,R_0,\ldots,R_n)\dint u$
exists in the Bochner sense and equals $F(R_0,\ldots,R_n)$.
\end{lemma}

In more suggestive notation this Lemma is a Fubini Theorem for the
product measure $dE du$, \ie~the product of the spectral measure
$\Phi$ and the Lebesgue measure on the half line $\R_{\ge 0}$.  Namely, using
the integral notation with respect to the resolution of the identity
it means
\[
\int_0^\infty \int_X f(u,\gl) \dint E(\gl)\dint u = \int_X \int_0^\infty
f(u,\gl) \dint u \dint E(\gl).
\]

\begin{proof} Let us first note that due to the integrability condition
\Eqref{EQIntegrability} and the Dominated Convergence Theorem
the function $F$ is indeed continuous.

To see the claimed Bochner integrability we note that $u\mapsto
f(u,R_0,\ldots, R_n)$ is continuous (\textit{cf.}~\cite[Prop.~4.10]{Tak:TOA}).
Furthermore, by the Spectral Theorem and the integrability condition
\Eqref{EQIntegrability} we have for the integral of the norm
\[
\int_0^\infty \|f(u,R_0,\ldots,R_n)\| \dint u
= \int_0^\infty \sup_{\gl\in X} |f(u,\gl)| \dint u < \infty.
\]
Thus the integral exists in the Bochner sense. 
Furthermore, for vectors $x,y\in\sH$ we have by continuity of the
Bochner integral
\begin{equation}\label{EQPfSpFub1}\begin{split}
  \binn{\int_0^\infty &f(u,R_0,\ldots, R_n) \dint u \; x, y}
     = \int_0^\infty \inn{f(u,R_0,\ldots, R_n) \, x, y} \dint u\\
     &= \int_0^\infty \int_X f(u,\gl) \dint E_{x,y}(\gl) \dint u.
\end{split}\end{equation}     
The latter integral is an ordinary product integral of the Radon
measure $E_{x,y}$ and the Lebesgue measure. Again by the integrability
condition \Eqref{EQIntegrability} we have
\[
\int_0^\infty \int_X | f(u,\gl)| \dint |E_{x,y}(\gl)| \dint u
  \le \|x\|\cdot \|y\|\cdot \int_0^\infty \sup_{\gl\in X} |f(u,\gl)|
  \dint u
  <\infty,
\]
hence Fubini's Theorem applies and we continue \Eqref{EQPfSpFub1} to
obtain
\[\begin{split}
\eqref{EQPfSpFub1}& = \int_X \int_0^\infty f(u,\gl) \dint u \dint E_{x,y}(\gl)\\
   &= \int_X F(\gl) \dint E_{x,y}(\gl) = \binn{F(R_0,\ldots, R_n)\, x,y}.
\end{split}\] 
This proves that indeed $\int_0^\infty f(u,R_0,\ldots,R_n) \dint u
= F(R_0,\ldots, R_n)$.
\end{proof}

\begin{theorem}[Operator Substitution Lemma]\label{TOSL} 
Let $R_0,\ldots, R_n$ be commuting selfadjoint positive operators as
in Sec.~\ref{ss2.1}.  Furthermore, let $f:\R_{\ge 0}^{n+1}=(\R_{\ge 0})^{n+1}\to\C$ be a
continuous function such that for each pair of positive real numbers
$0<C_1<C_2$ one has
\begin{equation}\label{EQIntegrability1}
\int_0^\infty \sup_{\substack{C_1\le s_j\le C_2\\ 0\le j\le n}}
  |f(us)| \dint u <\infty.
\end{equation}
Then for the functions
\[
F:\R_{>0}^{n+1}\ni s\mapsto \int_0^\infty f(u\cdot s) \dint u
\]
and
\[
G:\R_{>0}^n\ni \gl\mapsto \int_0^\infty f(u,u\gl_1,\ldots, u\gl_n)
\dint u
\]
we have the identity
\begin{multline*}
  \int_0^\infty f(uR_0,uR_1,\ldots, uR_n)\dint u
      = F(R_0,\ldots, R_n)\\
       = R_0\ii G(R_0\ii R_1,\ldots, R_0\ii R_n)
       = R_0\ii \int_0^\infty f(u, u R_0\ii R_1,\ldots, u R_0\ii R_n)
      \dint u.
\end{multline*}
Both integrals exist in the Bochner sense.
\end{theorem}
\begin{remark} We have formulated the Operator Substitution Lemma
multiplicatively. There is an obvious additive analogue
for integrals of the form 
$\int_\R h(x+T_0,x+T_1,\ldots,x+T_n) \dint x =
\int_\R h(x, x+T_1-T_0,\ldots, x+T_n-T_0) \dint x$ 
for commuting selfadjoint operators $T_0,\ldots, T_n$ and
appropriate functions $f\in C_0(\R^{n+1})$. We leave the details
to the reader.
\end{remark}
\begin{proof}
Put $g(u,s):=f(us)$, $0<u<\infty$, $s\in \R_{>0}^{n+1}$ and
$h(u,\gl):=f(u,u\gl_1,\ldots,u \gl_n)$, $\gl\in \R_{>0}^n$. 
Then by \Eqref{EQIntegrability1}
the Lemma \plref{LSpectralFubini} applies to both functions $g$ and $h$.
Furthermore, 
\[\begin{split}
   F(s) & = \int_0^\infty f(u s_0, u s_1,\ldots, u s_n) \dint u\\
        & = \int_0^\infty s_0\ii f(u, u s_0\ii s_1,\ldots, u s_0\ii
         s_n) \dint u
         = s_0\ii G(s_0\ii s_1,\ldots, s_0\ii s_n),
\end{split}\]        
and the claim follows.
\end{proof}

\begin{example}\label{ExOSL}
Let $\ga=(\ga_0,\ldots,\ga_p)\in \N^{p+1}$ 
be a multiindex. Then put
\begin{equation*} 
f(x_0,x_1,\ldots, x_p):= x_0^\nu\cdot\prod_{j=0}^p (1+x_j)^{-\ga_j-1},
\quad -1<\nu<|\ga|+p.
\end{equation*}
We show that $f$ satisfies the integrability condition
\Eqref{EQIntegrability1} of Theorem \plref{TOSL}.
Given $0<C_1<C_2$ then for $C_1\le s_j\le C_2$
and $0\le u\le 1$ we have
\[
  |f(us)|\le s_0^\nu\cdot u^\nu\le \textup{const} \cdot u^\nu,
\]
while for $u\ge 1$ we have
\[
  |f(us)|  = \Bigl| (s_0 u)^\nu \prod_{j=0}^p (s_j u)^{-\ga_j-1}\cdot
      \prod_{j=0}^p \bigl(\frac{(s_ju)}{1+s_j u)}\bigr)^{\ga_j+1}\Bigr| 
      \le \textup{const}\cdot |u|^{\nu-|\ga|-p-1},
\]     
hence the claim.

Inductively, one easily sees that for any multiindex $\ga$ the
function $\pl_s^\ga f(us)=u^\ga (\pl^\ga f)(us)$ also satisfies the
integrability condition \Eqref{EQIntegrability}.
\end{example}

\section{Tensor products and the Rearrangement Lemma}
\label{SecTPRL}

\subsection{Projective vs. maximal $C^*$--tensor product, 
the contraction map}

\subsubsection{Tensor product completions}
Let $\sA$ be a unital $C^*$--algebra. Denote by $\sA^{\otimes n+1}:=
\sA\otimes\ldots \otimes \sA$ the $(n+1)$--fold algebraic tensor
product. For elementary tensors we use the notations
$(a_0,\ldots,a_n)$ and $a_0\otimes \ldots \otimes a_n$ as synonyms. 
By
\[ 
 \mu_n: \sAn \to \sA, (a_0,\ldots,a_n)\mapsto a_0\cldots a_n
\]
we denote the multiplication map. 

We discuss the issue of extending the multiplication map to
tensor product completions of $\sAn$.
We denote by
$\sAn_\gamma$ the projective tensor product completion of $\sAn$,
\textit{cf., e.g.,} \cite{Gel1959}. That is $\sAn_\gamma$ is the completion of $\sAn$
with respect to the norm 
\[
\| x \|_\gamma = \inf \sum_i \| a_0^{(i)}\|\cldots \| a_{n}^{(i)}\|,
\]
where the infimum is taken over all representations of $x\in\sAn$ as
a finite sum $\sum_i (a_0^{(i)},\ldots, a_n^{(i)})$. $\sAn_\gamma$ is
a Banach--algebra. Moreover, the adjoint map is easily seen to be
continuous with respect to the norm $\|\cdot\|_\gamma$,
hence $\sAn_\gamma$ is a Banach--$*$--algebra.

Furthermore, let $\sAn_\pi$ be the maximal $C^*$--algebra
tensor product completion of $\sAn$ \cite[Sec.~IV.4]{Tak:TOA}. That
is $\sAn_\pi$ is the completion of $\sAn$ with respect to the norm
\[
\|x\|_\pi = \sup \|\gvr(x)\|,
\]
where $\gvr$ runs through all $*$--representations of $\sAn$.
$\|\cdot\|_\pi\le \|\cdot\|_\gamma$ and hence there is a natural
continuous $*$--homomorphism $\prgp: \sAn_\gamma\to \sAn_\pi$
whose range is dense.

Each of the two tensor products comes with a benefit and a curse and
these are mutually exclusive. The projective tensor product behaves
well in the sense that $\mu_n$ extends by continuity to a linear map
$\sAn_\gamma \to \sA$. It behaves badly in the sense that $\sAn_\gamma$,
although being a Banach $*$--algebra, is a $C^*$--algebra only in
trivial cases. 
On the other hand the $C^*$--algebra $\sAn_\pi$ behaves well in the
sense that it is $C^*$ and hence, \eg there is a continuous functional
calculus for commuting selfadjoint elements. It behaves badly in the
sense that the author does not know whether the multiplication map
$\mu_n$ extends by continuity to $\sAn_\pi$; in fact he suspects that
there exist interesting cases where it does not extend.  A poll among
available experts on tensor products was inconclusive.

Needless to say, for matrix algebras the algebraic tensor product is
already complete and there is no problem. Even in this seemingly
trivial case the results outlined below do have aspects which, 
to the best of our knowledge, seem to be new.

\subsubsection{The contraction map}\label{ssContraction}\label{ssTen2} 
We come to a crucial construction. 
For $a\in \sAn_\gamma$ and elements $b_1,\ldots, b_n\in \sA$ we write,
motivated by \cite[Lemma 6.2]{ConMos2011}, \cf \Eqref{EQIntro1}, suggestively
\begin{equation}\label{EQMult}
  a(b_1\cldots b_n):= \mu_n \bl a\cdot (b_1\otimes\ldots \otimes
  b_n\otimes  1_\sA)\br \in \sA,
\end{equation}
and call the result the contraction of $a$ by $b_1\oldots
b_n$.

Note that if $a=(a_0,\ldots,a_n)$ is an elementary tensor then
\begin{equation}\label{EQMult1}
  (a_0,\ldots,a_n)(b_1\cldots b_n)= a_0 \cdot b_1\cdot a_1\cldots
  a_{n-1}\cdot b_n \cdot a_n.
\end{equation}  

\Eqref{EQMult} induces a continuous map $\sAn_\gamma\times 
\sA^{\otimes n}_\gamma\to\sA$.
The whole discussion of this section circles around the problem of
extending \Eqref{EQMult} to a reasonable class of elements in
$\sAn_\pi$. The discussion would simplify considerably if
\Eqref{EQMult} would extend to a continuous map $\sAn_\pi\times
\sA_\pi^{\otimes n}\to \sA$. We do not know whether this is the case
as topologies on tensor products can behave notoriously pathologic.

%
\subsection{Smooth functional calculus on $\sAn_\gamma$}\label{SSSmoothFC}
%
For $a\in\sA$ put $A:= e^a$ and 
\begin{equation}\label{EQtensor1}\begin{split}
  a^{(j)} & = (1_\sA,\ldots,1_\sA,a,1_\sA,\ldots,1_\sA), \quad 0\le
  j\le n \quad \text{($a$ is in the $j$--th slot)}, \\
     \nabla_a^{(j)} &:= - a^{(j-1)} + a^{(j)}, \quad 1\le j \le n,\\
     \Delta_a^{(j)}&:= \exp(\nabla_a^{(j)}), \quad 1\le j\le n, \\
        & = (1_\sA,\ldots,1_\sA, A\ii,A,1_\sA,\ldots,1_\sA), \quad
        \text{($A\ii$ is in slot $j-1$)}.
\end{split}\end{equation} 
Note that slots are enumerated starting from $0$, so $a^{(0)}=a\otimes
1_\sA\otimes\ldots, a^{(1)}=1_\sA\otimes a\otimes
1_\sA\otimes\ldots$, etc.

The operators $a\pup 0;, \ldots, a\pup n;, \nabl 1;_a, \ldots, \nabl
n;_a, \modul 1;_a,\ldots, \modul n;_a$ commute. If $a$ is selfadjoint
then so are $a\pup j;, \nabl j;_a, \modul j;_a$. Furthermore, if $a$
is selfadjoint then $A\pup j;:=\exp(a\pup j;)$ is positive. 

The following simple identities are at the heart of the
Rearrangement Lemma:
\begin{align}
   A\pup j; & = (1_\sA,\ldots,1_\sA,A,1_\sA,\ldots,1_\sA) \nonumber\\
            & = (A A\ii,\ldots,AA\ii,A,1_\sA,\ldots)\nonumber \\
            & = A\pup 0; \modul 1;\cldots \modul j;,\quad  j\ge
            1, \label{EQTREAR5}  \\
  a\pup j;  &= a\pup 0;+\nabl 1;_a+\ldots + \nabl j;_a,
                \quad  j\ge 1. 
              \label{EQTREAR6}   
\end{align}

From now on assume that $a\in\sA$ is selfadjoint and let $\Phi:C(\spec
a)\to \sA, f\mapsto f(a)$ denote the spectral measure of $a$. The
$(n+1)$--fold tensor product, $\Phi_\pi$, is a $*$--isomorphism from
$C((\spec a)^{n+1})\simeq C(\spec a)_\pi^{\otimes n+1}$ onto the
unital $C^*$--subalgebra $C^*(I, a\pup 0;,\ldots, a\pup n;)$ of $\sAn_\pi$
generated by $a\pup 0;,\ldots, a\pup n;$. $\Phi_\pi$ is nothing
but the joint spectral measure of the commuting operators $a\pup 0;,\ldots,
a\pup n;$, \eg $\Phi_\pi(f)=f(a\pup 0;,\ldots, a\pup n;)$.
Furthermore, this $C^*$--algebra also contains the operators $ \nabl
1;_a, \ldots, \nabl n;_a,$ and $ \modul 1;_a,\ldots, \modul n;_a$. 

If we view the operators \Eqref{EQtensor1}
as elements of $\sAn_\gamma$ they still admit a joint
analytic functional calculus \cite{Tay1970}. We do not make use,
however, of this celebrated and somewhat demanding paper.
Instead we exploit the nuclearity of Fr{\'e}chet spaces of smooth
functions to establish a smooth functional calculus with values in
$\sAn_\gamma$.  To this end let $U\supset \spec a$ be an open set. Then the
algebra of smooth functions, $\cinf{U}$,  on $U$ with the usual
Fr{\'e}chet topology is known to be nuclear \cite[Sec.~51]{Tre:TVS}. Thus the
injective tensor product $\cinf{U}^{\otimes n+1}_\eps$ is isomorphic
to the projective tensor product $\cinf{U}^{\otimes n+1}_\gamma$. The
map $f_0\oldots f_n\mapsto \bl x\mapsto
f_0(x_0)f_1(x_1)\cldots f_n(x_n)\in\cinf{U^{n+1}}\br $ is known to
extend by continuity to an isomorphism
$\cinf{U}\otimes_\eps^{n+1}\simeq\cinf{U^{n+1}}$, hence by nuclearity
it also extends to an isomorphism $\cinf{U}\otimes_\gamma^{n+1}\simeq
\cinf{U^{n+1}}$.  The following commutative diagram summarizes these
considerations:
\[
\xymatrix{ \cinf{U^{n+1}}\ar[d]_{j_U}
\ar[r]^-{\Phi_\gamma} & \sAn_\gamma \ar[d]^{\prgp}\\
        C\bl(\spec a)^{n+1}\br \ar[r]^-{\Phi_\pi} &\sAn_\pi .}
\]
The horizontal arrows are continuous $*$--homomorphisms which on
elementary tensors are given by $f_0\oldots f_n\mapsto
f_0(a)\otimes\ldots \otimes f_n(a)$, the vertical arrows are 
$j_U(f):=f\restr{(\spec a)^{n+1}}\in C((\spec a)^{n+1})$
resp. the natural map from the projective to the maximal $C^*$--tensor
product. 
\mpar{take projective limit and define $\Phi_\gamma$ on
$C^\infty(\spec)$}

\begin{remark}[\textit{Schwartz functions, entire functions}]\label{ssTen4}
\sitem
We note in addition that a functional calculus for, say, Schwartz
functions can be set up in a more elementary way by the Fourier
transform. Namely, observe that for $\xi\in\R^{n+1}$ we have
\begin{equation*}
   \exp\bl i\xi_0 a\pup 0;+\ldots + i\xi_n a\pup n;\br
    = e^{i\xi_0 a}\otimes e^{i \xi_1 a}\oldots
    e^{i\xi_na},
\end{equation*}
and therefore, since $\|\cdot\|_\gamma$ is a cross-norm 
\begin{equation*}
   \Bigl\|
   \exp\bl i\xi_0 a\pup 0;+\ldots + i\xi_n a\pup n;\br
   \Bigr\|_\gamma
   = \bigl\| e^{i\xi_0 a}\bigr\|\cldots \bigl\|e^{i\xi_n a}\bigr\| \le
   1.
\end{equation*}
Thus for functions with integrable Fourier transform, e.~g. Schwartz
functions, we have 
\begin{equation}\label{EQten9}
  \Phi_\gamma(f):=f_\gamma(a\pup 0;,\ldots , a\pup n;)=
  \int_{\R^{n+1}} \widehat f(\xi) 
    \exp\bl i \inn{ \xi, a\pup \cdot;}\br \,\dsl\xi,
\end{equation}
where $\inn{ \xi, a\pup\cdot;}$ is an abbreviation for
$\xi_0 a\pup 0;+\ldots + \xi_n  a\pup n;$, and this integral
converges in $\sAn_\gamma$ in the Bochner sense. 

\sitem Finally, for an entire function $f(z)=\sum_{\ga} f_\ga z^\ga$
in $n+1$ variables $z=(z_0,\ldots, z_n)$, of course, 
$\Phi_\gamma(f)=f_\gamma(a\pup 0;,\ldots,a\pup n;)$ is given by the convergent series
obtained by inserting $a\pup j;$ for $z_j$.
\end{remark}

\begin{theorem}\label{TProduct}
Let $\sA$ be a unital $C^*$--algebra and let $a\in\sA$ be a
selfadjoint element. Furthermore, let $U\supset \spec a$ be an open
neighborhood of $\spec a$. 

\sitem There is a unique continuous unital $*$--homomorphism
$\Phi_\gamma:\cinf{U^{n+1}}\simeq \cinf{U}^{\otimes n+1}_\gamma \to
\sAn_\gamma$ sending $f_0\otimes\ldots \otimes f_n$ to
$f_0(a)\otimes_\gamma\ldots\otimes_\gamma f_n(a)$. $\Phi_\gamma$
is compatible with the spectral measure of $a$ in the sense that
$\prgp(\Phi_\gamma(f))= f(a\pup 0;,\ldots, a\pup n;)$.
We therefore write $f_\gamma(a\pup0;,\ldots, a\pup
n;)$ for $\Phi_\gamma(f)$.

For $f\in\cinf{U^{n+1}}$ the element 
$f_\gamma(a\pup 0;,\ldots, a\pup n;)=\Phi_\gamma(f)\in\sAn_\gamma$
depends only on $f$ in an arbitrarily small open neighborhood
of $(\spec a)^{n+1}$. In particular, for $f$ one may therefore choose
a Schwartz function $\tilde f$ with $\tilde f\equiv f$ in such a
neighborhood. Then $f_\gamma(a\pup 0;,\ldots, a\pup n;)=
\tilde f_\gamma(a\pup 0;,\ldots, a\pup n;)$ which can be calculated by
the integral \Eqref{EQten9}.
\sitem The map
\begin{multline*}
\cinf{U^{n+1}}\times \sA^{\otimes n}_\gamma\longrightarrow \sA,\\
     (f, b_1\oldots b_{n})\mapsto
     \mu_n\bl f_\gamma(a\pup 0;,\ldots, a\pup n;) \cdot
       (b_1\oldots b_{n}\otimes 1_\sA)\br
\end{multline*}
is the unique continuous linear map sending
$(f_0\oldots f_n, b_1\otimes\ldots \otimes b_{n})$
to $f_0(a)\cdot b_1\cdot f_1(a) \cdot b_2\cldots b_n\cdot f_n(a)$.
\end{theorem}
For the last map we therefore use, as defined 
in Sec.~\ref{ssContraction}, the shorthand notation
$f_\gamma(a\pup 0;,\ldots,a\pup n;)(b_1\cldots b_n)$.
\begin{proof} This Theorem just summarizes what we explained in 
so far in Sec~\ref{SSSmoothFC}. The last claim in 1. follows
from a simple partition of unity argument. \mpar{1. direct ref}
\end{proof}

\begin{remark}\label{Rem1}
We note that for a Schwartz function $f\in \sS(\R^{n+1})$
we have as a Bochner integral
\[\begin{split}
   f_\gamma(&a\pup 0;,\ldots , a\pup n;)(b_1\cldots b_n)=
   \int_{\R^{n+1}} \widehat f(\xi) \exp_\gamma\bl
      i\inn{ \xi, a\pup\cdot;\br}\br (b_1\cldots b_n) \,\dsl\xi\\
    & = 
    \int_{\R^{n+1}} \widehat f(\xi) e^{i\xi_0 a} b_1 e^{i \xi_1 a}
    b_2\cldots b_n e^{i\xi_n a} \,\dsl\xi,
\end{split}\]    
resp. for $f\in\sS(\R^n)$
\[\begin{split}
   f_\gamma(&\nabl 1;_a,\ldots , \nabl n;_a)(b_1\cldots b_n)=
      \int_{\R^n} \widehat f(\xi) \exp_\gamma\bl i\inn{ \xi , \nabl
    \cdot;_a}\br (b_1\cldots b_n) \,\dsl\xi\\
    & = 
    \int_{\R^n} \widehat f(\xi) e^{-i\xi_1 a} b_1 e^{i(\xi_1-\xi_2)a}
    b_2\cldots b_n e^{i\xi_n a} \,\dsl\xi.
\end{split}\]    
Here we have used
\begin{equation*}
   i\xi_1 \nabl 1;_a+\ldots +i\xi_n\nabl n;_a
   = 
   -i\xi_1 a\pup 0; +i (\xi_1-\xi_2)a\pup 1;+\ldots+
    i\xi_n a\pup n;,
\end{equation*}
which follows from \Eqref{EQtensor1}.
\end{remark}

\subsection{The Rearrangement Lemma}
We will need versions of Lemma \plref{LSpectralFubini} and Theorem
\plref{TOSL} for the smooth functional calculus in the
Banach--$*$--algebra $\sAn_\gamma$. For this the integrability
conditions \Eqref{EQIntegrability} and \eqref{EQIntegrability1}
have to be assumed for all partial derivatives of the involved
function.

\begin{theorem}[Smooth Operator Substitution Lemma]\label{TOSLSmooth} 
Let $\sA$ be a unital $C^*$--algebra and let $a\in\sA$ be a
selfadjoint element. Put $A:=e^a$ and let $U\supset \spec A$ be an
open neighborhood of $\spec A$.  
\sitem Let $f:\R_{\ge 0}\times U^{n+1}\to \C$ be a smooth function satisfying 
the following integrability
condition: for each compact subset $K\subset U$ and each multiindex
$\ga\in \N^{n+1}$
\begin{equation}\label{EQIntegrabilitySmooth}
   \int_0^\infty \sup_{\gl\in K} |\pl^\ga_\gl f(u,\gl)| \dint u <\infty.
\end{equation}   
Then $F(\gl):= \int_0^\infty f(u,\gl) \dint u$ defines a smooth function
on $U^{n+1}$, the integral $\int_0^\infty f_\gamma(u,A\pup
0;,\ldots,A\pup n;)\dint u$ exists as a Bochner integral with values in
$\sAn_\gamma$ and the integral equals $F_\gamma(A\pup 0;,\ldots,A\pup n;)$.

\sitem 
Let $f:\R_{\ge 0}^{n+1}\to\C$ be a smooth function such that for each pair of
positive real numbers $0<C_1<C_2$ and each multiindex 
$\ga\in \N^{n+1}$
\begin{equation}\label{EQIntegrabilitySmooth1}
\int_0^\infty \sup_{\substack{C_1\le s_j\le C_2\\ 0\le j\le n}}
|u^{|\ga|} (\pl^\ga f)(us)| \dint u <\infty.
\end{equation}
Then for the smooth functions $F(s)= \int_0^\infty f(u\cdot s) \dint u$
and $G(\gl) = \int_0^\infty f(u,u\gl_1,\ldots, u\gl_n) \dint u$
as in Theorem \ref{TOSL}  one has
\[\begin{split}
  \int_0^\infty &f_\gamma (u A\pup 0;,\ldots, u A\pup n;)\dint u
      = F_\gamma(A\pup 0;,\ldots,A \pup n; )\\
      & = A\ii G_\gamma(\modul 1;, \modul 1; \cdot\modul 2;,\ldots,\modul
      1;\cldots \modul n;)\\
      & = A\ii \int_0^\infty f_\gamma(u, u \modul 1;,u \modul
      1;\cdot\modul 2;,\ldots, u \modul 1;\cldots \modul n;) \dint u.
\end{split}\]
Both integrals exist in the Bochner sense in $\sAn_\gamma$ resp.
$\sA^{\otimes n}_\gamma$.
\end{theorem}
\begin{proof}
\begin{numbered}
\sitem
The integrability condition guarantees that the integral 
$\int_0^\infty f(u,\cdot) \dint u$ converges
as a Bochner integral with values in the Fr{\'e}chet space
$\cinf{U^{n+1}}$. Thus $F$ is smooth and integration 
commutes with continuous linear maps. Denote by $\Phi_\gamma$ the
$\sAn_\gamma$--valued spectral measure of $A\pup 0;,\ldots,A\pup n;$
according to Theorem \ref{TProduct}. Then
\[\begin{split}
   \int_0^\infty &f_\gamma(u,A\pup 0;,\ldots,A\pup n;)\dint u
      = \int_0^\infty \Phi_\gamma(f(u,\cdot)) \dint u\\
      & = \Phi_\gamma \Bl \int_0^\infty f(u,\cdot) \dint u \Br
      = F_\gamma(A\pup 0;,\ldots, A\pup n;).
\end{split}\]      
\sitem Let $g(u,s):=f(us), h(u,\gl):= f(u,u\gl_1,\ldots, u\gl_n)$ as
in the proof of Theorem \ref{TOSL}. Then by the integrability
condition the proven first part applies to both functions $g$ and $h$
and, taking into account the relations \Eqref{EQTREAR5}, the claim
follows as in the proof of Theorem \ref{TOSL}. 
\end{numbered}
\end{proof}

Together with Theorem \ref{TProduct} we obtain
as an immediate consequence:

\pagebreak[3]
\begin{cor}[Rearrangement Lemma]\label{TREAR}
Let $f_0,\ldots,f_p:\R_{\ge 0}\to\C$ be smooth functions such that
$f(x_0,\ldots,x_p):=\prod_{j=0}^p f_j(x_j)$ satisfies the
integrability condition \Eqref{EQIntegrabilitySmooth1} of the 
Smooth Operator
Substitution Lemma \ref{TOSLSmooth}. Furthermore, let $a$ be a selfadjoint
element of the unital $C^*$--algebra $\sA$, put $A:=e^a$.  Moreover,
denote by $\Delta\pup\cdot;, \nabla\pup\cdot;$ the operators defined
in \Eqref{EQtensor1}.  Then for $b_1,\ldots,b_p\in \sA$
\begin{equation}\label{EQTREAR1}\begin{split}
\int_0^\infty &f_0(u A)\cdot b_1\cdot f_1(u A)\cldots b_p\cdot f_p(uA)
\dint u \\ 
   &= A\ii 
     \int_0^\infty f_\gamma(u, u\modul 1;, u\modul 1;
        \modul 2;,\ldots, u\modul 1;\cldots\modul p;) \dint u 
        (b_1\cldots b_p)\\
   & = A\ii F_\gamma( \modul 1;, \modul 1; \modul 2;,\ldots, \modul 1;\cldots\modul p;) 
       (b_1\cldots b_p),
\end{split}\end{equation}
where the smooth function $F(s_1,\ldots,s_p)$ is
\[
  F(s) = \int_0^\infty f_0(u)\cdot f_1(us_1)\cldots f_p(u s_p) \dint u.
\]
\end{cor}



\begin{example}\label{ExREAR}
We continue Example \ref{ExOSL} and put
\[\begin{split}
   f_0(x)&:= x^\nu (1+x)^{-\ga_0-1},\\
   f_j(x)&:= (1+x)^{-\ga_j-1}.
\end{split}\]
Then Corollary \ref{TREAR} applies and we recover
the Rearrangement Lemma of Connes-Moscovici \cite[Lemma
6.2]{ConMos2011}.
\end{example}

\subsection{Noncommutative Taylor expansion in terms of divided
differences}
\label{SSecMagnus}

Given selfadjoint elements $a,b$ of the unital $C^*$--algebra $\sA$.
We recast the noncommuatative Taylor expansion formula 
(\textit{cf., e.g.}, \cite[Sec.~6.1]{ConMos2011})
for $\exp(a+b)$ in light of the functional calculus summarized in
Theorem \ref{TProduct} and the Genocchi-Hermite formula
\Eqref{EQHermite} for divided differences.  The main facts about
divided differences are summarized in Appendix \ref{SecDD} below. 

The expansional formula for the exponential function reads
\begin{equation}\label{EQExpansion}
e^{a+b} = e^a + \sum_{n=1}^\infty
     \int_{0\le s_n\le \ldots \le s_1\le 1} 
       e^{(1-s_1)a}\cdot b\cdot e^{(s_1-s_2)a}\cdot b\cldots b\cdot
       e^{s_na} \dint s.
\end{equation}
The integrand equals, \cf Remark \ref{ssTen4} and Remark \ref{Rem1},
\[
    \exp_\gamma\bl (1-s_1) a\pup 0;+(s_1-s_2) a\pup 1;+\ldots+ s_n a\pup
    n;\br (b\cldots b).
\]
Applying the Genocchi-Hermite formula \Eqref{EQHermite}
to the exponential function we have
\begin{multline*}
 \int_{0\le s_n\le \ldots \le s_1\le 1} 
    \exp_\gamma 
    {\textstyle\bl (1-s_1) a\pup 0;+(s_1-s_2) a\pup 1;+\ldots+ s_n a\pup
    n;\br} \dint s_1\ldots \dint s_n \\
    = [a\pup 0;,\ldots , a\pup n;] \exp_\gamma.
\end{multline*}
In other words the general term in the expansion formula
\Eqref{EQExpansion} can be reinterpreted as follows: take the
commuting selfadjoint operators $a\pup 0;,\ldots,a\pup n;$ and insert
them into the multivariable function $x\mapsto \DDots x_0,x_n;\exp$,
the $n^{\text th}$ divided difference of the exponential function. 
Then contract with the $n$-fold tensor product $b\otimes\ldots \otimes
b$.

Therefore, the formula \Eqref{EQExpansion} may be rewritten in
the very compact way
\begin{align}
e^{a+b} &= \sum_{n=0}^\infty \bl[a\pup 0;,\ldots,a\pup n;] \exp_\gamma
\br(b\cldots b)\label{EQExp1a}\\
        &= \sum_{n=0}^\infty e^a \bl[0,\nabl 1;_a, 
        \nabl 1;_a+\nabl 2;_a,\ldots,\nabl 1;_a+\ldots+\nabl n;_a] 
        \exp_\gamma \br(b\cldots b).\label{EQExp1b}
\end{align}
In the second line we have used the functional equation of $\exp$, the
homogeneity of the divided differences (\textit{cf.}~\Eqref{EQDD1}), and the
relations \Eqref{EQTREAR6}.  In a different context it was also
observed in \cite{BaxBru2011} that the expansion formula
\Eqref{EQExpansion} can be interpreted in terms of the
Genocchi-Hermite formula.
We obtain a straightforward generalization of \Eqref{EQExp1a},
\eqref{EQExp1b} to arbitrary smooth functions.

\begin{prop}\label{PMagnus} Let $a\in\sA$ be selfadjoint.  Then for a smooth
function $f$ in a neighborhood of $\spec a$ the Taylor
expansion of $f(a+b)$ for selfadjoint $b\sim 0$ is given by
\[
f(a+b) \sim_{b\to 0} \sum_{n=0}^\infty 
\bl[a\pup 0;,\ldots,a\pup n;]  f_\gamma \br(b\cldots b).
\]
\end{prop}
\begin{remark}\sitem
Note that if $\sA=\C$ and hence $a, b$ are real numbers then
$\sAn$ is canonically isomorphic to $\C$ and under this isomorphism
$\bl[a\pup 0;,\ldots,a\pup n;]  f_\gamma \br(b\cldots b)$
corresponds to $\frac 1{n!} f^{(n)}(a) b^n$, see
\Eqref{EQDD4}, and the Proposition just gives  the ordinary Taylor
formula.

\sitem
The formula in Prop. \ref{PMagnus} is equivalent to the noncommutative
Taylor expansion formula derived in \cite{Pay2011} in the context of
formal power series. This expansion was in fact discovered earlier
by Daletskii \cite{Dal1990}. We plan to discuss such expansions and
its relations to a noncommutative Newton interpolation formula in
more detail in the near future.
\end{remark}

\begin{proof} W.~l.~o.~g.~we may assume that $f$ is a Schwartz function on
$\R$, \cf Theorem \ref{TProduct}, 1. Write
\[
  f(a+b)=\int_{\R} \widehat f(\xi) e^{i\xi (a+b)} \,\dsl \xi.
\]  
Then apply the expansion formula \Eqref{EQExpansion} to the
exponential term
\[
e^{i\xi(a+b)} = e^{i\xi a}+
     \sum_{n=1}^\infty (i\xi)^n 
     \bl[i\xi a\pup 0;,\ldots,i\xi a\pup n;]  \exp_\gamma \br(b\cldots b).
\]
Noting that $(i\xi)^n\widehat f(\xi)= \widehat{f^{(n)}}(\xi)$
the $n$-th term (with the $b$'s omitted) equals
\begin{multline*}
 \int_{0\le s_n\le \ldots \le s_1\le 1} 
 f_\gamma^{(n)} \bl 
 (1-s_1) a\pup 0;+(s_1-s_2) a\pup 1;+\ldots+ s_n a\pup
    n;\br \dint s_1\ldots ds_n \\
    = [a\pup 0;,\ldots , a\pup n;] f_\gamma,
\end{multline*}    
where Genocchi-Hermite's formula 
\Eqref{EQHermite} was used.
\end{proof}

\begin{example} Let $a(s,t)\in\sA$ be a smooth selfadjoint family
with $a(0,0)=a$.
Put $\delta_1 a:=\pl_s\restr{s=0}a(s,0), \delta_2 a:=\pl_t\restr{t=0}
a(0,t)$, and $\delta_1\delta_2 a:=\pl_s\pl_t\restr{s=t=0}a(s,t)$.
Then
\begin{align}
\pl_s\restr{s=0} f(a(s,0))
    & = ([a\pup 0;, a\pup 1;] f_\gamma)(\delta_1 a)\\
\pl_s\pl_t\restr{s=t=0} f(a(s,t)) & =
     ( [a\pup 0;, a\pup 1;] f_\gamma)(\delta_1\delta_2 a)+ \\
      &\qquad
         + ([a\pup 0;,a\pup 1;,a\pup 2;]f_\gamma)(\delta_1 a\delta_2 a+
         \delta_2 a\delta_1 a).\nonumber
\end{align}         
Taking into account \Eqref{EQExp1b} we obtain for the exponential
function
\begin{align}
e^{-a}\pl_s\restr{s=0} e^{a(s,0)}
    & = ([0, \nabl 1;_a] \exp_\gamma)(\delta_1 a)\\
e^{-a}\pl_s\pl_t\restr{s=t=0} e^{a(s,t)} & =
     ( [0,\nabl 1;_a] \exp_\gamma)(\delta_1\delta_2 a)+\\
      &\qquad
         + ([0,\nabl 1;_a, \nabl 1;_a+\nabl 2;_a]\exp_\gamma)(\delta_1 a\delta_2 a+
         \delta_2 a\delta_1 a).\nonumber 
\end{align}         
Note that
\begin{align}
   [0,s]\exp     & = \frac{e^s-1}{s},\\
   [0,s,s+t]\exp & = \frac{e^{s+t}s+t-e^s(s+t)}{s t (s+t)}.
\end{align}
One should compare this to \cite[(21)]{ConTre2011},
\cite[(167)-(169)]{ConMos2011}, and \cite[Lemma 5.1]{FarKha2013}.
\end{example}

\subsection{Expansion formulas for $\nabla_a$}
\label{SSecNabla}

Recall from \Eqref{EQtensor1}
$\nabla_a:=\nabla_a^{(1)}=-a\otimes 1_\sA + 1_\sA\otimes a\in \sA\otimes \sA$.
To expand $f(\nabla_{a+b})$ we therefore  have to apply the expansion 
of Proposition \ref{PMagnus} in the algebra $\tilde\sA:=\sA\otimes_\gamma\sA$. 
Denote for $c\in\tilde\sA$, analogously to 
\Eqref{EQtensor1},
\[\begin{split}
  \tilde c^{(j)} & = (1_\stA,\ldots,1_\stA,c,1_\stA,\ldots,1_\stA), 
       \quad 0\le j\le n \quad \text{($c$ is in the $j$--th slot)}, \\
       \tilde\nabla_a^{(j)} &:= \bl\nabla_a\br^{(j)}
          = (1_\stA,\ldots,1_\stA,\nabla_a,1_\stA,\ldots,1_\stA)
              , \quad 0\le j \le n.
\end{split}\] 
\begin{lemma}\label{LNabla1}
Let $f\in\sS(\R^{n+1})$ be a Schwartz function, let
$b_j'\otimes b_j''\in\stA, j=1,\ldots,n$, and let $x\in\sA$ be given.
Note that $f_\gamma(\tnabl 0;_a,\ldots,\tnabl n;_a)\in\stAn_\gamma$.
After contraction with $(b_1'\otimes b_1'')\otimes \ldots\otimes
(b_n'\otimes b_n'')$ one obtains an element of $\stA$ which
can be contracted further with $x\in\sA$ to an element of $\sA$.
For this element we have
\begin{align*}
\bl f_\gamma(&\tnabl 0;_a,\ldots,\tnabl n;_a)(b_1'\otimes b_1''\cldots b_n'\otimes
b_n'')\br(x)\nonumber\\
&= f_\gamma(-a\pup 0; +a\pup n+1;, -a\pup 1; + a\pup n+2;,\ldots, -a\pup n;
   +a\pup 2n+1;)\\
   &\qquad (b_1'\cldots b_n'\cdot x\cdot b_1''\cldots b_n'')\nonumber\\
& = f_\gamma(\nabl 1;_a+\ldots+ \nabl n+1;_a,
      \nabl 2;_a+\ldots+ \nabl n+2;_a,\ldots,
      \nabl n+1;_a+\ldots+ \nabl 2n+1;_a)\\
      &\qquad 
      (b_1'\cldots b_n'\cdot x\cdot b_1''\cldots b_n'').\nonumber
\end{align*}      
\end{lemma}
\begin{proof}
This follows from a straightforward calculation:
\[\begin{split}
\bl f_\gamma(&\tnabl 0;_a,\ldots,\tnabl n;_a)(b_1'\otimes b_1''\cldots b_n'\otimes
b_n'')\br(x)\nonumber\\
    & =
      \int_{\R^{n+1}} \widehat f(\xi)
         \bl e^{-i\xi_0 a}\otimes e^{i\xi_0 a} b_1'\otimes
         b_1''\otimes\ldots
         \otimes b_n'\otimes b_n'' e^{-i\xi_n a}\otimes e^{i\xi_n a}
         \br(x) \,\dsl \xi\\
    & =
      \int_{\R^{n+1}} \widehat f(\xi)
         e^{-i\xi_0 a} b_1' e^{-i\xi_1 a} b_2' \cldots
         b_n' e^{-i\xi_n a} x e^{i\xi_0 a} b_1'' \cldots b_n'' 
         e^{i\xi_n a} \,\dsl \xi\\
    & = f_\gamma(-a\pup 0; +a\pup n+1;, -a\pup 1; + a\pup n+2;,\ldots, -a\pup n;
            +a\pup 2n+1;)\\
     &\qquad \qquad (b_1'\cldots b_n'\cdot x\cdot b_1''\cldots b_n'').\nonumber
     \qedhere
\end{split}\]
\end{proof}

This Lemma and the expansion \plref{PMagnus} allow to expand 
$f(\nabla_{a+b})(x)$ in principle to any order, although the
combinatorics becomes tedious. We note the expansion up to order $2$,
\textit{cf.}~\cite[Lemma 4.11 and Lemma 4.12]{ConMos2011}.

\begin{prop}\label{PNablaExpansion} 
Let $a, x\in\sA$ be selfadjoint.
Then for a Schwartz function $f\in\sS(\R)$ the 
Taylor expansion up to order $2$
of $f(\nabla_{a+b})(x)$ for selfadjoint $b\sim 0$ 
is given by
\[\begin{split}
  f(\nabla_{a+b})(x)  = &   f(\nabla_a)(x) \\
      &  -([\nabl 1;_a+\nabl 2;_a, \nabl 2;_a]f_\gamma)(b\cdot x)
        + ([\nabl 1;_a+\nabl 2;_a, \nabl 1;_a]f_\gamma)(x\cdot b)\\
      & + ([\nabl 1;_a+\nabl 2;_a+\nabl 3;_a , \nabl 2;_a+\nabl 3;_a,
      \nabl 3;_a]f_\gamma)(b\cdot b \cdot x)\\
      & + ([\nabl 1;_a, \nabl 1;_a + \nabl 2;_a,\nabl 1;_a+\nabl 2;_a
           +\nabl 3;_a]f_\gamma)(x\cdot b \cdot b)\\
      & + ([\nabl 1;_a + \nabl 2;_a, 
          \nabl 1;_a +\nabl 2;_a+\nabl 3;_a,
          \nabl 2;_a+\nabl 3;_a]f_\gamma)(b\cdot x \cdot b)\\
      & + ([\nabl 1;_a + \nabl 2;_a, \nabl 2;_a,
            \nabl 2;_a+\nabl 3;_a]f_\gamma)(b\cdot x \cdot b).
\end{split}\]
\end{prop}
The two variable functions involved in the linear term are 
\begin{equation}
-[s+t,t]f= -\frac{f(s+t)-f(t)}{s},\quad
[s+t,s]f=\frac{f(s+t)-f(s)}{t},
\end{equation}
this should be compared to \cite[(134)]{ConMos2011}.
\begin{proof} One just has to apply Prop. \ref{PMagnus} to $f(\nabla_{a+b})$
in the algebra $\stA$ and apply the previous Lemma. We do the
calculation for the linear term and leave the second order term to the
interested reader. 
\[\begin{split}
  ([\tnabl 0;_a, \tnabl 1;_a] &f_\gamma)(\tilde \nabla_b)(x)
    = ([\tnabl 0;_a, \tnabl 1;_a] f_\gamma)(-b\otimes 1_\sA +
    1_\sA\otimes b)(x)\\
   = & ([-a\pup 0; + a\pup 2;,-a\pup 1; + a\pup 3;]f_\gamma)(- b\cdot x\cdot
  1 + 1\cdot x\cdot b)\\
   = & -([\nabl 1;_a+ \nabl 2;_a, \nabl 2;_a]f_\gamma)( b\cdot x)
       +([\nabl 1;_a+ \nabl 2;_a, \nabl 1;_a]f_\gamma)( x\cdot b).
       \qedhere
\end{split}\]
\end{proof}

\begin{cor} Let $\gvf$ be a tracial state on $\sA$. Then, for
 selfadjoint elements $a, b, x, y\in\sA$ we have
\begin{multline*}
 \frac{d}{d\eps}\restr{\eps=0}
   \varphi\bl f(\nabla_{a+\eps b})(x)y \br \\
   = -\varphi\bl b ([\nabl 1;_a, -\nabl 2;_a]f_\gamma) (x\cdot y)\br
      +\varphi\bl b ([-\nabl 1;_a, \nabl 2;_a]f_\gamma) (y\cdot x)\br.
\end{multline*}
\end{cor} 
Note that 
\begin{equation}
  -[s, -t]f=\frac{f(-t)-f(s)}{s+t},\quad
  [-s, t]f =\frac{f(t)-f(-s)}{s+t}.
\end{equation}
This should be compared to \cite[(131)]{ConMos2011},
where $f$ is assumed to be even and hence $-[s,-t]f=[-s,t]f=\mlfrac
f(t)-f(s)/s+t;$.
\begin{proof} Using the previous Proposition we calculate
\begin{align*}
 \frac{d}{d\eps}\restr{\eps=0}
   &\varphi\bl f(\nabla_{a+\eps b})(x)y\br\\
 = & -\varphi \bl([-a\pup 0; + a\pup 2;,-a\pup 1; + a\pup 2;]f_\gamma)(
 b\cdot x)\cdot y \br\\
  &\qquad + \varphi \bl([-a\pup 0; + a\pup 2;,-a\pup 0; + a\pup 1;]f_\gamma)(
 x\cdot b)\cdot y \br\\
 = &  -\varphi\bl b ([- a\pup 2;+a\pup 1;, - a\pup 0; + a\pup 1;]f_\gamma) (x\cdot y)\br
      \\
      &\qquad 
 +\varphi\bl b ([-a\pup 1;+a\pup 0;, -a\pup 1;+a\pup 2;]f_\gamma)
 (y\cdot x)\br,
\end{align*}
and the result follows in view of \Eqref{EQtensor1} and the fact that
divided differences are symmetric functions of their arguments.
\end{proof}

\section{The functions occurring in the Rearrangement Lemma
for the modular curvature}
\label{SecFunctions}


\subsection{The Mellin transform of $(1+x)^{-m-1}$}
\label{ssIntegral}
By a contour integral argument
\cite[3.123]{Tit:TF} the Mellin transform of $x\mapsto (1+x)\ii$
is given by
\[
 \int_0^\infty x^{z-1} \mlfrac 1/1+x; \dint x = 
     \mlfrac \pi/\sin\pi z;, \quad 0<\Re z<1,
\]
and integration by parts yields
\[
  \int_0^\infty x^{z-1} \mlfrac 1/(1+x)^{m+1}; \dint x
  = \frac{\ffak z-1,m;}{m!} 
       \mlfrac \pi/\sin\pi z;.
\]
Since $\mlfrac 1/\sin \pi z;$ decays exponentially on vertical lines
we conclude that the functions $x\mapsto (1+x)^{-m-1}$ are given
by the inversion formula
\[
   (1+x)^{-m-1} = \int_{\Re z=\ga} x^{-z} 
   \frac{\ffak z-1,m;}{m!} 
       \mlfrac \pi/\sin\pi z;\dint z
\]
for $0<\Re \ga<1$.

\subsection{The functions $\fM(s,m)$ and $\fH(s,m)$}
\label{ssFunctions}

Given $p\in \Z_{\ge 1}$, a multiindex $\ga\in\N^{p+1}$
and $s_j>0, j=0,\ldots, p,$ put 
\begin{align}
\fM(s,z)& := \int_0^\infty
x^{|\ga|+p-1-z} \cdot \prod_{j=0}^p (1+s_j x)^{-\ga_j-1} \,dx,
\quad -1<\Re z < |\ga|+p, \label{EQ2.2a}\\
        &= \int_0^\infty x^{z}\cdot \prod_{j=0}^p (x+ s_j)^{-\ga_j-1} \,dx,
                 \label{EQ2.2b}
\end{align}
where the second line is obtained by changing variables $x\mapsto
x\ii$. Furthermore,
\begin{equation*}
\fH(s',z):= \fM((1,s'),z),\quad s'=(s_1,\ldots, s_p).
\end{equation*}
We are mainly interested in integral values of $z$.  The integrals
\Eqref{EQ2.2a}, \eqref{EQ2.2b} converge absolutely for 
$-1<\Re z<|\ga|+p$. So $z=m\in\Z$ may take the values
$0,1,\ldots,|\ga|+p-1$.  The function 
$\fM(\cdot,z)$ is $(- |\ga|-p+z)$--homogeneous, that is
\begin{equation}\label{EQ2.3}
  \fM(\gl s,z) = \gl^{-|\ga|-p+z} \fM(s,z),
\end{equation}  
as is seen by changing variables from $\gl x$ to $x$. Therefore, 
scaling $s_0$ gives
\begin{equation*}
  \fM( s,z) = s_0^{-|\ga|-p+z} \fH(s'/s_0,z).
\end{equation*}  

$\fM(s,m)$ and $\fH(s,m)$ can be expressed in terms of closed formulas
involving divided differences and differentiations:

\begin{prop}\label{PIFormula} For a multiindex $\ga=(\ga_0,\ldots,\ga_p)\in
\N^{p+1}$ and $s=(s_0,\ldots, s_p)$ with $s_j>0$ let
$(u_0,\ldots,u_{|\ga|+p})$ be the tuple with $u_0=\ldots
=u_{\ga_0}=s_0, u_{\ga_0+1}=\ldots=u_{\ga_0+\ga_1+1}=s_1,
\ldots, u_{|\ga|+p-1-\ga_p}=\ldots= u_{|\ga|+p}=s_p$.\footnote{In other
words, the tuple $u$ consists of $\ga_0+1$ copies of $s_0$, $\ga_1+1$
copies of $s_1$ etc.} Furthermore, let 
$\ga':=(0,\ga_1,\ldots,\ga_p)$.
Then for $m\in \{0, 1,\ldots,|\ga|+p-1\}$
\begin{align}
\fM(s,m)   & = (-1)^{m+|\ga|+p-1} \DDots u_0,u_{|\ga|+p}; \id^m \log
                  \label{EQPIFormula1}\\
      & = \mlfrac (-1)^{m+|\ga|+p-1}/\ga!; \pl_s^\ga 
                    \DDots s_0,s_p; \id^m \log
                  \label{EQPIFormula2}.
\end{align}                    
Here, $\id^m$ stands for the function $x\mapsto x^m$ and 
$\DDots y_0,y_n; f$ stands for the divided difference of the
function $f$ with respect to the variables $y_0,\ldots, y_n$.

If $m\in \{0, 1,\ldots,|\ga'|+p-1\}$ then also
\begin{align}
 \fM(s,m)  & = \mlfrac (-1)^{|\ga'|+p-1-m}/\ga!;
     \Bffak \sum_{k=1}^p s_k \pl_{s_k} +|\ga|+p-1-m,\ga_0;\cdot
                  \label{EQPIFormula3}\\
      & \qquad \qquad \cdot \pl_s^{\ga'} \DDots s_0,s_p; \id^m \log.
                    \nonumber
\end{align}                    
Here, $\ffak \sum_{k=1}^p s_k \pl_{s_k} +|\ga|+p-1-m,\ga_0;$
is the differential operator $\sum_{k=1}^p s_k \pl_{s_k} +|\ga|+p-1-m$
inserted into the falling factorial polynomial 
$\ffak a,\ga_0;= a\cdot (a-1)\cldots (a-n+1)$.

Consequently, for $\fH$ and $m\in\{0,1,\ldots,|\ga'|+p-1\}$
we have the following formula which only involves partial derivatives
in the variables $s_1,\ldots, s_p$
\begin{align}
\fHp \ga,p;(s',m)   
    & = \fM((1,s'),m),\quad s'=(s_1,\ldots, s_p)\nonumber \\
    & = \mlfrac (-1)^{|\ga'|+p-1-m}/\ga!;
        \Bffak \sum_{k=1}^p s_k \pl_{s_k} +|\ga|+p-1-m,\ga_0;\cdot
                  \label{EQPIFormula3a}\\
    & \qquad \qquad \cdot \pl_s^{\ga'} [1,s_1,\ldots,s_p] \id^m \log.
                    \nonumber
\end{align}
\end{prop}

Recall that divided differences are explained in Appendix \ref{SecDD}
below. For more on the falling factorials see Sec.~\ref{ss3.2} below.

\begin{proof}
We start with distinct positive variables $t_0,\ldots, t_q; q:=|\ga|+p$. 
Then by \Eqref{EQ2.2b}
\[
   \fMp 0,q;(t,m) 
      = \int_0^\infty x^{m} \prod_{j=0}^q (x+ t_j)\ii \dint x.
\]
The integrand is a rational function of degree $m-q-1\le -2$.
Therefore, it has a partial fraction decomposition
\[
      x^{m} \prod_{j=0}^q (x+ t_j)\ii 
         = \sum_{k=0}^q A_k \; (x+t_k)\ii,
\]
with
$\sum\limits_{k=0}^q A_k  = 0$. The $A_k$ are explicitly
given by
\[
 A_k  = (-t_k)^{m} \prod_{j=0, j\not=k}^q (t_j-t_k)\ii
      = (-1)^{m+q}\; t_k^{m}\; \prod_{j=0, j\not=k}^q (t_k-t_j)\ii.
\]
Thus we find
\begin{equation}\label{EQ3.5.1}\begin{split}
  \fMp 0,q;(t,m) &= -\sum_{k=0}^q A_k \log t_k \\
  &  = (-1)^{m+q-1} \sum_{k=0}^q \Bl \prod_{j=0, j\not=k}^q (t_k-t_j)\ii\Br
     t_k^{m} \log t_k \\
     & = (-1)^{m+q-1} \DDots t_0,t_q; \id^{m} \log.  
\end{split}\end{equation}     
In the last equation \Eqref{EQDD2} was used.
By continuity this formula also holds for not necessarily distinct
variables $t_0,\ldots, t_q$. Hence by \Eqref{EQDD4}
\[\begin{split}
\fM (s,m) & = (-1)^{m+|\ga|+p-1} \DDots s_0^{\ga_0+1},s_p^{\ga_p+1};
                  \id^{m} \log\\
                  & = \mlfrac (-1)^{m+|\ga|+p-1}/\ga!; \pl_s^\ga 
             \DDots s_0,s_p; 
                  \id^{m} \log,
\end{split}\]                  
thus \Eqref{EQPIFormula1} and \Eqref{EQPIFormula2} are proved. 

The proof of the remaining claims about the formulas involving falling
factorials is postponed to the Appendix \ref{AppB}. 
\end{proof}

\begin{remark}  For general $z\not\in\Z$ one may calculate $\fMp 0,q;(t,z)$
similarly. From the partial fraction decomposition 
\[\begin{split}
      \prod_{j=0}^q (x+ t_j)\ii 
         &= \sum_{k=0}^q A_k \; (x+t_k)\ii,\\
     A_k  &= (-1)^{q}  \prod_{j=0, j\not=k}^q (t_k-t_j)\ii,
\end{split}\]
and Sec.~\ref{ssIntegral} we infer
\[\begin{split}
  \fMp 0,q;(t,z) &= \mlfrac -\pi/\sin\pi z;\sum_{k=0}^q A_k t_k^z\\
  &  = \mlfrac (-1)^{q-1}\pi/\sin\pi z; \sum_{k=0}^q \Bl \prod_{j=0, j\not=k}^q (t_k-t_j)\ii\Br
     t_k^{z} \\
     & =\mlfrac (-1)^{q-1}\pi /\sin\pi z; \DDots t_0,t_q; \id^z.  
\end{split}\]     
Taking the limit $z\to m\in \Z$ one obtains again \Eqref{EQ3.5.1}.
\end{remark}

\section{Examples}
\label{SecExamples}

Recall from \Eqref{EQ2.2a}, \eqref{EQ2.2b} and Proposition \ref{PIFormula}
that for $s_j>0$ and $m\in \{0,1,\ldots,|\ga|+p-1\}$
\begin{align}
\fH(s,m)& := \int_0^\infty x^{|\ga|+p-1-m} \cdot (1+x)^{-\ga_0-1} \cdot
         \prod_{j=1}^p (1+s_j x)^{-\ga_j-1} \dint x,\\
        & = \int_0^\infty x^{m} \cdot (1+x)^{-\ga_0-1}\cdot
          \prod_{j=1}^p (x+ s_j)^{-\ga_j-1} \dint x\\
        & = (-1)^{m+|\ga|+p-1}\cdot [1^{\ga_0+1}, s_1^{\ga_1+1}, \ldots,
          s_p^{\ga_p+1}] \id^m\log.
\end{align}
The recursion formula \Eqref{EQDD1}, the Leibniz rule \Eqref{EQDD5},
and the substitution rule \Eqref{EQSubstitution} lead to a large
variety of recursion formulas for the functions $\fH$. We will discuss
here the case of one and two variable functions and in particular
compare the two variable case to the examples listed at the end of
\cite{ConMos2011}.

\subsection{One variable functions} 
\label{SSecOneVar}
\subsubsection{} 
From \Eqref{EQPIFormula3a} we infer
\begin{equation}
\fHp {0,0},1;(s):= \fHp {0,0},1;(s,0):=[1,s]\log= \mlfrac \log s/s-1;=:\cL_0(s).
\end{equation}
Note that if we substitute $s=\exp(u)$ this function becomes
\begin{equation}
   \mlfrac u/e^u-1; = \sum_{j=0}^\infty \mlfrac B_j/j!; u^j,
\end{equation}
which is the generating function for the Bernoulli numbers. The fact
that by Proposition \plref{PIFormula} all the functions $\fH$ are
ultimately expressed in terms of the function $\mlfrac \log s/s-1;$ is
one of the ``conceptual explanations'' the formidable formulas (3) and
(4) in \cite{ConMos2011} are ``begging'' for.

Applying Proposition \plref{PIFormula} we find if
$m\in\{0,1,\ldots,\ga_1\}$
\begin{align}
  \fHp \ga,1; (s,m) &= \mlfrac (-1)^{|\ga_1|+m}/\ga!;
  \ffak D_s+|\ga|-m,\ga_0; \pl_s^{\ga_1} \mlfrac s^m \log s/s-1;\\
     & =   
     \mlfrac (-1)^{\ga_1+m}/\ga!;
     s^{m-\ga_1} \pl_s^{\ga_0} s^{|\ga|-m} \pl_s^{\ga_1} 
       \mlfrac s^m \log s/s-1;\\
     & =   
     \mlfrac (-1)^{\ga_1+m}/\ga!;
      \pl_s^{\ga_1} s^{m} \pl_s^{\ga_0} 
       \mlfrac s^{\ga_0} \log s/s-1;,\label{EQHalph3}
\end{align}
resp. for $m=0$
\begin{align}
   \fHp \ga,1;(s)&:=\fHp \ga,1; (s,0) = (-1)^{|\ga|} [1^{\ga_0+1}, s^{\ga_1+1}] \log
   \\
         & = (-1)^{|\ga|} [1^{\ga_0}, s^{\ga_1+1}] \cL_0,\quad
         \cL_0(s):=[1,s]\log\\
         & = 
     \mlfrac (-1)^{\ga_1}/\ga!;
     s^{-\ga_1} \pl_s^{\ga_0} s^{|\ga|} \pl_s^{\ga_1} 
       \mlfrac \log s/s-1;\label{EQHalph0}\\
       & = \mlfrac (-1)^{\ga_1}/\ga!; \pl_s^{|\ga|}
       \mlfrac s^{\ga_0} \log s/s-1;,\label{EQHalph1}
\end{align}
where the substitution rule \Eqref{EQSubstitution} was used.
For the equalities \Eqref{EQHalph3} and \eqref{EQHalph1} \cf~\Eqref{EQ3.3}.

\subsubsection{} 
\label{ssEx2}
We note the special case
\begin{equation}
  \begin{split}
\cL_m(s)&:=\fHp {0,m},1;(s,m) = \fHp {m,0},1;(s,0)
    = (-1)^m [1^{m+1},s]\log \\
    &= (-1)^m [1^m,s]\cL_0 = \mlfrac 1/m!; 
   \pl_s^m \mlfrac s^m \log s/s-1;\\
   & = \mlfrac (-1)^m/(s-1)^{m+1};
      \Bl \log s - \sum_{j=1}^m \mlfrac (-1)^{j-1}/j; (s-1)^j\Br
  \end{split}
\end{equation}
which was called ``modified Logarithm'' in \cite[Sec.~3 and
6]{ConTre2011}.

We list the first few functions explicitly. 
\begin{align}
\fHp {1,0},1; (s) & = \cL_1(s)= -[1,s]\cL_0 = - \mlfrac \log s -s +1 /
(s-1)^2;, \nonumber\\
\fHp {0,1},1; (s) & = \mlfrac s \log s -s +1 /s (s-1)^2;, \\
\fHp {1,1},1; (s) & = [1^2,s^2]\log = -\pl_s \fHp {1,0},1;(s)= 
   -\mlfrac 2 s \log s -s^2+1/(s-1)^3 s;.\nonumber
\end{align}

\subsection{Two variable functions}
\label{SSecTwoVar}
Instead of the clumsy $\fHp \ga,2;((a,b),0)$ we write
$\fHp \ga,2;(a,b)$. By \Eqref{EQPIFormula3a} and the substitution rule
\Eqref{EQSubstitution} we have
\begin{equation}
  \begin{split}
\fHp {\ga},2;(a,b)
    & =  (-1)^{|\ga|+1} [1^{\ga_0+1}, a^{\ga_1+1}, b^{\ga_2+1}]\log, \\
    & = (-1)^{|\ga|+1} [1^{\ga_0}, a^{\ga_1+1}, b^{\ga_2+1}]\cL_0\\
    & = (-1)^{|\ga|+1} \frac{1}{b-a}\Bl
      [1^{\ga_0}, a^{\ga_1}, b^{\ga_2+1}]\cL_0
      -[1^{\ga_0}, a^{\ga_1+1}, b^{\ga_2}]\cL_0\Br\\
      & = \frac{(-1)^{|\ga|+\ga_0+1}}{\ga_1!\ga_2!}\pl_a^{\ga_1}\pl_b^{\ga_2}  
      \frac 1{b-a} \Bl \cL_{\ga_0}(b)-\cL_{\ga_0}(a)\Br.
  \end{split}
\end{equation}
Thus in the special case $\ga_1=\ga_2=0$ we immediately obtain a
simple formula expressing two variable functions in terms of one
variable modified logarithms:
\begin{equation}\label{EQOneTwoReduction}
  \fHp {r,0,0},2;(a,b) = \frac {-1}{b-a}\bl \cL_r(b)-\cL_r(a)\br.
\end{equation}

\subsection{Comparison with the explicit formulas in \cite{ConMos2011}}

For two variable functions $\fHp \ga,2;(s)$ let us compare
our results to the explicit formulas given at the end
of \cite{ConMos2011}. We denote the function $H$ introduced there by
$H^{\textup CM}$. Then for the two variable functions we have by
definition 
$\HCM_{\ga_0+1,\ga_1+1,\ga_2+1}(a,b)=
\fHp \ga,2;(a,b)$.

In \cite{ConMos2011} the following formulas are given explicitly.
In the resp. first lines we list the formulas as stated in loc.~cit.,
in the resp. second lines we cancel common factors and write them
as a sum of fractions involving $\log(a), \log(b)$ plus terms
which do not contain logarithms. As a helper the open source computer
algebra system \texttt{Maxima} was used.
\begin{scriptsize}
\begin{align*}
\HCM_{1,1,1}(a,b) & = \mlfrac (-1+b) \log(a) - (-1+a)\log(b)/
    (-1+a)(-1+b)(-a+b); \\
      & = \mlfrac \log(a)/(a-1)(b-a); -
                 \mlfrac \log(b)/(b-1)(b-a);,\\
\HCM_{1,2,1}(a,b) & = \mlfrac (-1+b)\bl (-1+a)(a-b)+a(1-2a+b)\log(a)\br
         + (-1+a)^2 a \log (b)/ (-1+a)^2a(a-b)^2(-1+b); \\
        &=\mlfrac (b-2a+1)\log(a)/(a-1)^2(b-a)^2;
                +\mlfrac \log(b)/(b-1)(b-a)^2;
                  - \mlfrac 1/(b-a)(a-1)a;,\\
\HCM_{2,1,1}(a,b) & = \mlfrac (-1+b)^2\log(a) +
                      (-1+a)\bl (a-b)(-1+b) - (-1+a)\log(b)\br / 
                        (-1+a)^2(a-b)(-1+b)^2;\\
           &= - \mlfrac\log(a)/(b-a)(a-1)^2; + 
           \mlfrac  \log(b)/(b-1)^2(b-a);+ 
           \mlfrac 1/(b-1)(a-1);, 
\end{align*}
\begin{align*}
&\HCM_{2,2,1}(a,b) \\ & = 
\mlfrac (-1+b)\Bl (-1+a)(a-b)\bl 1+a^2-(1+a)b\br 
    +a(-1+3a-2b)(-1+b)\log(a)\Br -(-1+a)^3a\log(b)/
    (-1+a)^3a(a-b)^2(-1+b)^2;
    \\
             & = -\mlfrac (2b -3a +1)\log(a)/(b-a)^2(a-1)^3;
                 - \mlfrac \log(b)/(b-1)^2(b-a)^2;
                 + \mlfrac (a+1)b-a^2-1/(b-1)(b-a)(a-1)^2a;, 
\end{align*}
\begin{align*}
\HCM_{3,1,1}(a,b) & =
\mlfrac
(-1+a)(5+a(-3+b)-3b)(a-b)(-1+b)-2(-1+b)^3\log(a)+2(-1+a)^3\log(b)/2
(-1+a)^3(a-b)(-1+b)^3;
\\
     & = \mlfrac \log(a)/(b-a)(a-1)^3;
         -\mlfrac \log(b)/(b-a)(b-1)^3;
         +\mlfrac (a-3)b-3a+5/2(b-1)^2(a-1)^2;.
\end{align*}      
\end{scriptsize}

From the resp. second lines we see immediately that
$\HCM_{1,1,1}(a,b)=-[1,a,b] \log$ and that 
$\HCM_{1,2,1}(a,b)=-\pl_a \HCM_{1,1,1}(a,b)$.

To $\HCM_{2,1,1}$ and $\HCM_{3,1,1}$ we can apply 
\Eqref{EQOneTwoReduction} and obtain
\[\begin{split}
  \HCM_{2,1,1}(a,b) &= \frac {-1}{b-a} \bl\cL_1(b)-\cL_1(a)\br\\
  \HCM_{3,1,1}(a,b) &= \frac {-1}{b-a} \bl\cL_2(b)-\cL_2(a)\br.
\end{split}\]
Alternatively, one may employ the formulas in Proposition
\plref{PIFormula} and indeed one verifies
\[\begin{split}
  \HCM_{2,1,1}(a,b) & = \pl_s\restr{s=1} [s,a,b]\log\\
     & = - (a\pl_a +b\pl_b+2) [1,a,b]\log \\
     & =   (a\pl_a +b\pl_b+2) \HCM_{1,1,1}(a,b),\\
\HCM_{3,1,1}(a,b) & = \fHp{2,0,0},2;(a,b) 
      = -\frac 12 \ffak a\pl_a + b\pl_b+3,2; [1,a,b]\log\\
     & = -\frac 12 (a\pl_a +b\pl_b+3) (a\pl_a +b\pl_b+2) [1,a,b]\log \\
     & = \frac 12  (a\pl_a +b\pl_b+3) \HCM_{2,1,1}(a,b).
\end{split}\]     
Similarly,
\[\begin{split}     
\HCM_{2,2,1}(a,b) & = \fHp{1,1,0},2;(a,b)
               = -\pl_a \fHp{1,0,0},2;(a,b)\\
                      & =  -\pl_a \HCM_{2,1,1}(a,b).
\end{split}\]     

\subsection{Conclusion}
The possibilities to produce such formulas are endless. All these
formulas can be obtained, of course, by performing partial fraction
decompositions on the integrand of \Eqref{EQ2.2a} resp.
\Eqref{EQ2.2b}. However, the calculus of finite differences with its
various rules provides a convenient framework which allows to obtain
the formulas in a mechanical way.

\appendix
\section{Divided differences}
\label{SecDD}

Divided differences have their origin in interpolation theory; they
can be traced back to Newton. Although being standard textbook
material in numerical analysis, let us give a very quick summary here;
for a recent survey see \cite{deB2005}, a classical reference is
\cite{MT1951}. In the sequel all functions are assumed to be smooth.

\subsection{}
Let $f$ be a smooth function on a real interval $I$ and
let $x_0, x_1,\ldots$ a priori distinct points in $I$. Then one
defines recursively the \emph{divided differences}
\begin{equation}\label{EQDD1}\begin{split}
     \DD x_0; f &:= f(x_0),\\
     \DDots x_0,x_n; f &:= \mlfrac 1/x_0-x_n;\bl \DDots x_0,x_{n-1}; f 
         - \DDots x_1,x_n; f\br.
\end{split}\end{equation}
The first few divided differences are therefore
\begin{small}
\begin{align*}
    [x_0,x_1]f & = \mlfrac f(x_0)/(x_0-x_1);+ \mlfrac
                        f(x_1)/(x_1-x_0);,\\
    [x_0,x_1,x_2] f  & = \mlfrac f(x_0)/(x_0-x_1)(x_0-x_2);+
                         \mlfrac f(x_1)/(x_1-x_0)(x_1-x_2);+
                         \mlfrac f(x_2)/(x_2-x_0)(x_2-x_1);,
\end{align*}
\end{small}
and by induction one shows the explicit formula
\begin{equation}\label{EQDD2}
     \DDots x_0,x_n; f  = \sum_{k=0}^n f(x_k) 
            \prod_{j=0, j\not=k}^n (x_k-x_j)\ii,
\end{equation}
resp. the Genocchi-Hermite integral formula \cite[Sec.~1.6]{MT1951},
\cite[Sec.~9]{deB2005}\footnote{%
According to the historical remarks in \cite[Sec.~9]{deB2005} the
formula is due to Genocchi who communicated it to Hermite in a letter.
}
\begin{equation}\label{EQHermite}\begin{split}
   &\DDots x_0,x_n; f = \int_{\sum\limits_{j=0}^n s_j =1, s_j>0}
       f^{(n)}\bl\sum\limits_{j=0}^n s_j x_j\br \dint s_1\dots ds_n\\
       &= \int_{0\le t_n\le\ldots \le t_1\le 1}
       f^{(n)}\bl (1-t_1)x_0+\ldots+(t_{n-1}-t_n)x_{n-1}+t_nx_n\br
       \dint t_1\dots dt_n.
\end{split}\end{equation}
If $f$ is even analytic, e.~g.~if $f$ is already an interpolation
polynomial, and if $\gamma$ is a closed curve in the domain of $f$ encircling the
points $x_0,\ldots,x_n$ exactly once then by the Residue Theorem
and \Eqref{EQDD2} we have \cite[Sec.~1.7]{MT1951}
\begin{equation}\label{EQDD3}
   \DDots x_0,x_n; f = \mlfrac 1/2\pi i; \oint_\gamma f(\gz) \prod_{j=0}^n (\gz-x_j)\ii
   \dint \gz.
\end{equation}   

\subsection{The confluent case}
\label{ssA2}
From the right hand sides of \Eqref{EQHermite} and \Eqref{EQDD3} we
see that $\DDots x_0,x_n;f$ is a smooth (analytic) function of the
variables $x_0,\ldots, x_n$. Therefore, one uses these formulas to
extend the divided differences to the confluent case of repeated
arguments. Thus, for any $x_0,\ldots,x_n\in I$, regardless of being
pairwise distinct or not, $\DDots x_0,x_n;f$ is a smooth
(analytic) symmetric function of its arguments.

The divided differences can be calculated quite efficiently from the
recursion system \Eqref{EQDD1} and with some care this can also be
extended to the confluent case \cite[1.8]{MT1951}.
Alternatively, there is a differentiation formula relating a divided
difference with repeated arguments to one with distinct arguments.
This is obtained by differentiating by the parameters under the
integral in \Eqref{EQDD3} or in Genocchi-Hermite's formula \Eqref{EQHermite}.

To explain this consider a multiindex $\ga=(\ga_0,\ldots,
\ga_n)\in \N^{n}$ and $x_0,\ldots, x_n\in I$. We write
$\DDots x_0^{\ga_0+1},x_n^{\ga_n+1};f$ for the divided
difference $\DDots u_0,u_{|\ga|+n};f$ where the tuple
$(u_0,\ldots,u_{|\ga|+n})$ contains exactly $\ga_0+1$ copies of $x_0$,
$\ga_1+1$ copies of $x_1$ etc. From \Eqref{EQDD3} we infer
\cite[Sec.~1.8]{MT1951}
\begin{align}
   [x_0^{\ga_0+1}&,\ldots , x_n^{\ga_n+1}]f
   = \frac 1{2\pi i} \oint_\gamma f(\gz) \prod_{j=0}^n 
  (\gz -x_j)^{-\ga_j-1} \dint\gz \nonumber \\
  & = \mlfrac 1/\ga!; \pl_x^\ga \DDots x_0,x_n;f \label{EQDD4}\\
  & =  \sum_{k=0}^n \mlfrac 1/\ga_k!; \pl_{x_k}^{\ga_k}
     \Bl f(x_k) \prod_{j=0, j\not=k}^n
     (x_k-x_j)^{-\ga_j-1}\Br.\nonumber
\end{align}
Recall that we are using the multiindex notation for partial
derivatives and factorials, \cf Sec.~\ref{ssNot}.

\subsection{Leibniz rule}
The Leibniz rule for the divided difference of
a product \cite[Sec.~4]{deB2005}
\begin{equation}\label{EQDD5}
  \DDots x_0,x_n; (f\cdot g) = \sum_{j=0}^n \DDots x_0,x_j; f \cdot 
  \DDots x_j,x_n; g,
\end{equation}
can be used to deduce interesting recursion formulas. Namely, taking 
$g=\id$ or $\id^2$ we find
\[\begin{split}
  \DDots x_0,x_n;(\id f) &= 
    x_0 \cdot\DDots x_0, x_n; f + \DDots x_1,x_n; f\\
  \DDots x_0,x_n;(\id^2 f)&= 
    x_0^2 \cdot\DDots x_0, x_n; f +\\
    &\qquad +(x_0+x_1)\cdot\DDots x_1,x_n; f
    + \DDots x_2,x_n; f.
\end{split}\]
Of course, this can be extended to arbitrary powers, \cf
\cite[Sec.~1.31]{MT1951}.

\subsection{Substitution rule}\label{ssSubstitution} 
The following
generalization of the recursion scheme \Eqref{EQDD1} can be proved
easily by induction (\textit{cf.}~\cite[Prop. 11]{Jam14}). Given
$y_0,\ldots, y_p$ put $g(x):=[y_0,\ldots,y_p,x]f$. Then
\begin{equation}\label{EQSubstitution}
  \DDots x_0,x_q;g = [y_0,\ldots, y_p, x_0,\ldots, x_q]f.
\end{equation}  

\section{Homogeneous functions and totally characteristic differential
operators}
\label{AppB}

We muse a little about the totally characteristic derivative $x\pl_x$, 
certainly a little more than is barely necessary to see the formulas
\Eqref{EQPIFormula3} and \Eqref{EQPIFormula3a}.

\subsection{}\label{ss3.2} 
We will make frequent use of the
\emph{rising} and \emph{falling} factorials (aka Pochhammer
symbol) for which we adopt D. Knuth's notation 
\cite[p. 50]{Knu97}\footnote{He actually attributes it to A.~Capelli (1893)
and L.~Toscano (1939).}
\begin{align}
    \rfak a, n; &:= a\cdot (a+1)\cldots (a+n-1), &\rfak a, 0; &:= 1,
    \label{EQ3.1a} \\
    \ffak a, n; &:= a\cdot (a-1)\cldots (a-n+1), &\ffak a, 0; &:= 1. 
    \label{EQ3.1b}
\end{align}
Furthermore we denote by $D_x=x\pl_x$ the totally characteristic
derivative with respect to the variable $x$. For a polynomial $p\in\C[t]$ 
we write $p(\pl_x)$ resp. $p(D_x)$ for $\pl_x$ resp. $D_x$ inserted
into the indeterminate $t$. In particular, e.~g. , $\rfak D_x+k,n;$
stands for $D_x$ inserted into the polynomial $\rfak t+k,n;\in\C[t]$.

As an example we note the formula
\begin{equation}\label{EQ3.3}
\begin{split}
 x^{a} \pl_x^n x^{b} \pl_x^m( x^c \cdot)
 & = x^{a+b+c-n-m}\cdot\ffak D_x+b+c-m,n;\cdot \ffak D_x+c,m;\\
   & =  x^{a+b-n} \pl_x^m x^{n+m-b} \pl_x^n( x^{b+c-m} \cdot), 
      \quad n,m\in \N,\; a,b,c\in\C.
\end{split}     
\end{equation}
This can be seen in a lot of ways. The obvious way is to expand the
l.~h.~s.~via the Leibniz' rule and then apply the Binomial Theorem. A
much quicker way is to note that we have $D_x x^z = z \cdot x^z$ for any
complex number $z$ and that for any such $z$
\begin{equation*}
 \begin{split}
  x^{a} \pl_x^n x^b \pl_x^m  x^{c+z} &= 
 \ffak z+b+c-m,n;\cdot \ffak z+c,m; \cdot x^{z+a+b+c-n-m} \\
    & =  x^{a+b-n} \pl_x^m x^{n+m-b} \pl_x^n x^{b+c-m+z},
\end{split}
\end{equation*}
and since the l.~h.~s.~and the r.~h.~s.~of \Eqref{EQ3.3} are
polynomials in $D_x$, they must be equal. \Eqref{EQ3.3} contains
\Eqref{EQHalph1} as special case.

An immediate consequence of \Eqref{EQ3.3} is the fact that the
family of differential operators 
$x^{n-k} \pl_x^n (x^{k} \cdot), \quad k,n\in \Z, 0\le k\le n$, is commuting.

We mention another important property of totally characteristic
operators which is useful if one deals with the first integrand
\Eqref{EQ2.2a} which is a function of $s_jx$. 
Namely, if $p(t)\in\C[t]$ is a complex polynomial and $f$ a 
differentiable function then
\begin{equation}
  p(D_x) f(xs) = p(D_s) f(xs) = \bl p(D) f\br (xs).
\end{equation}  

\subsection{Homogeneous functions}\label{ss3.3}
If $\gG\subset \R^q$ is an
open cone we denote by $\cP^a(\gG)=\cP^a$ the space of smooth
functions on $\gG$ which are $a$--homogeneous, that is
\begin{equation*}
  f(\gl\cdot \xi) = \gl^a\cdot f(\xi).
\end{equation*}
Recall from \Eqref{EQ2.3} that the function $\fM(\cdot,z)$ is $(-|\ga|-p+z)$--homogeneous.
$a$--Homogeneous functions satisfy \emph{Euler's} identity
\begin{equation}\label{EQEuler}
  \sum_{j=1}^q D_j f = a\cdot f.
\end{equation}  
Consequently, on $\cP^a$ we may replace $D_1$ by $-\sum_{j=2}^q D_j + a$. 

\subsection{The basic function $b(x)=\mlfrac 1/1+x;$} 
Using the above
mentioned rules the following formulas for the basic function
$b(x)=\mlfrac 1/1+x;$ occurring in the integral \Eqref{EQ2.2a} can
easily be derived\footnote{Of course, they can also be derived by
brute force.}
\begin{equation}
  \begin{split}
 \pl b^l & = - l\cdot b^{l+1},\\
 (D+l) b^l & = l\cdot  b^{l+1},\\
 \pl^n b & = (-1)^n\cdot n!\cdot  b^{n+1},\\
 \rfak D+1,n;b &= \ffak D+n,n; b  = n!\cdot  b^{n+1},  \\
 x^{n-k} \pl^n x^k b(x) &= \ffak D_x+k,n; b(x) 
    = (-1)^{n-k} \cdot n!\cdot  x^{n-k} \cdot b(x)^{n+1}.
  \end{split}  
\end{equation}

\subsection{Proof of \Eqref{EQPIFormula3} and \Eqref{EQPIFormula3a} }
Recall from \Eqref{EQ2.3} that $\fM(\cdot,m)$ is
$(-|\ga|-p+z)$--homogeneous. Therefore, for $\ga=(\ga_0,\ga')$
we infer from \Eqref{EQPIFormula2} and \Eqref{EQ3.3} for
$m\in\{0,1,\ldots,|\ga'|+p-1\}$
\begin{equation}\label{EQ551}\begin{split}
   s_0^{\ga_0} \fM(s,m)  
     & =
       \mlfrac (-1)^{\ga_0}/\ga!; s_0^{\ga_0} \pl_{s_0}^{\ga_0}
          \fMp \ga',p;(s,m),\\
          & = \mlfrac (-1)^{\ga_0}/\ga!;\, \ffak D_{s_0},\ga_0;\,
          \fMp \ga',p;(s,m),
\end{split}\end{equation}          
and since $\fMp \ga',p;(\cdot,m)$ is $(-|\ga'|-p+m)$--homogeneous
we may replace $\ffak D_{s_0},\ga_0;$ by
\begin{equation}\label{EQ552}\begin{split}
\Bffak -\sum_{k=1}^p D_{s_j}-|\ga'|-p+m,\ga_0;
  & = (-1)^{\ga_0} 
     \Brfak \sum_{k=1}^p D_{s_j}+|\ga'|+p-m,\ga_0;\\
  & = (-1)^{\ga_0} 
     \Bffak \sum_{k=1}^p D_{s_j}+|\ga|+p-1-m,\ga_0;.
\end{split}\end{equation}
From \Eqref{EQ551} and \eqref{EQ552} the remaining claims
of Proposition \plref{PIFormula} follow.\qed



\bibliography{mlbib} 

\def\cprime{$'$}
\providecommand{\bysame}{\leavevmode\hbox to3em{\hrulefill}\thinspace}
\providecommand{\MR}{\relax\ifhmode\unskip\space\fi MR }
\providecommand{\MRhref}[2]{%
  \href{http://www.ams.org/mathscinet-getitem?mr=#1}{#2}
}
\providecommand{\href}[2]{#2}
\begin{thebibliography}{\textsc{BhMa12}}

\bibitem[\textsc{BaBr11}]{BaxBru2011}
\textsc{B.~J.~C. Baxter} and \textsc{R.~Brummelhuis}, \emph{Functionals of
  exponential {B}rownian motion and divided differences}, J. Comput. Appl.
  Math. \textbf{236} (2011), no.~4, 424--433. \texttt{arXiv:1006.1996
  [math.NA]}, \MR{2843028 (2012j:91126)}

\bibitem[\textsc{BhMa12}]{BhuMar2012}
\textsc{T.~A. Bhuyain} and \textsc{M.~Marcolli}, \emph{The {R}icci flow on
  noncommutative two-tori}, Lett. Math. Phys. \textbf{101} (2012), no.~2,
  173--194. \texttt{arXiv:1107.4788 [hep-th]}, \MR{2947960}

\bibitem[\textsc{CoMo14}]{ConMos2011}
\textsc{A.~Connes} and \textsc{H.~Moscovici}, \emph{Modular curvature for
  noncommutative two-tori}, J. Amer. Math. Soc. \textbf{27} (2014), no.~3,
  639--684. \texttt{arXiv:1110.3500 [math.QA]}, \MR{3194491}

\bibitem[\textsc{CoTr11}]{ConTre2011}
\textsc{A.~Connes} and \textsc{P.~Tretkoff}, \emph{The {G}auss-{B}onnet theorem
  for the noncommutative two torus}, Noncommutative geometry, arithmetic, and
  related topics, Johns Hopkins Univ. Press, Baltimore, MD, 2011, pp.~141--158.
  \texttt{arXiv:0910.0188 [math.QA]}, \MR{2907006}

\bibitem[\textsc{Dal90}]{Dal1990}
\textsc{Y.~L. Daletski{\u\i}}, \emph{The noncommutative {T}aylor formula and
  functions of triangular operators}, Funktsional. Anal. i Prilozhen.
  \textbf{24} (1990), no.~1, 74--76. \MR{1052274 (91d:47013)}

\bibitem[\textsc{dB05}]{deB2005}
\textsc{C.~de~Boor}, \emph{Divided differences}, Surv. Approx. Theory
  \textbf{1} (2005), 46--69. \texttt{arXiv:0502036 [math.CA]}, \MR{2221566
  (2006k:41001)}

\bibitem[\textsc{FaKh12}]{FarKha2012}
\textsc{F.~Fathizadeh} and \textsc{M.~Khalkhali}, \emph{The {G}auss-{B}onnet
  theorem for noncommutative two tori with a general conformal structure}, J.
  Noncommut. Geom. \textbf{6} (2012), no.~3, 457--480. \texttt{arXiv:1005.4947
  [math.OA]}, \MR{2956317}

\bibitem[\textsc{FaKh13}]{FarKha2013}
\bysame, \emph{Scalar curvature for the noncommutative two torus}, J.
  Noncommut. Geom. \textbf{7} (2013), no.~4, 1145--1183.
  \texttt{arXiv:1110.3511 [math.QA]}, \MR{3148618}

\bibitem[\textsc{Gel59}]{Gel1959}
\textsc{B.~R. Gelbaum}, \emph{Tensor products of {B}anach algebras}, Canad. J.
  Math. \textbf{11} (1959), 297--310. \MR{0104162 (21 \#2922)}

\bibitem[\textsc{Jam}]{Jam14}
\textsc{G.~Jameson}, \emph{Interpolating polynomials and divided differences},
  downloaded from http://www.maths.lancs.ac.uk/$\sim$jameson/interpol.pdf April
  17, 2014.

\bibitem[\textsc{Knu97}]{Knu97}
\textsc{D.~E. Knuth}, \emph{The art of computer programming. {V}ol. 1},
  Addison-Wesley, Reading, MA, 1997, Fundamental algorithms, Third edition [of
  MR0286317]. \MR{3077152}

\bibitem[\textsc{LeMo15}]{LesMos2015}
\textsc{M.~Lesch} and \textsc{H.~Moscovici}, \emph{Modular curvature and morita
  equivalence},  \texttt{arXiv:1505.00964 [math.QA]}.

\bibitem[\textsc{Mag54}]{Mag1954}
\textsc{W.~Magnus}, \emph{On the exponential solution of differential equations
  for a linear operator}, Comm. Pure Appl. Math. \textbf{7} (1954), 649--673.
  \MR{0067873 (16,790a)}

\bibitem[\textsc{MT51}]{MT1951}
\textsc{L.~M. Milne-Thomson}, \emph{The {C}alculus of {F}inite {D}ifferences},
  Macmillan and Co., Ltd., London, 1951. \MR{0043339 (13,245c)}

\bibitem[\textsc{Pay11}]{Pay2011}
\textsc{S.~Paycha}, \emph{Noncommutative formal {T}aylor expansions and second
  quantised regularised traces}, Combinatorics and physics, Contemp. Math.,
  vol. 539, Amer. Math. Soc., Providence, RI, 2011, pp.~349--376. \MR{2790317
  (2012e:58048)}

\bibitem[\textsc{Rud91}]{Rud91}
\textsc{W.~Rudin}, \emph{Functional analysis}, second ed., International Series
  in Pure and Applied Mathematics, McGraw-Hill, Inc., New York, 1991.
  \MR{1157815 (92k:46001)}

\bibitem[\textsc{Tak02}]{Tak:TOA}
\textsc{M.~Takesaki}, \emph{Theory of operator algebras. {I}}, Encyclopaedia of
  Mathematical Sciences, vol. 124, Springer-Verlag, Berlin, 2002, Reprint of
  the first (1979) edition, Operator Algebras and Non-commutative Geometry, 5.
  \MR{1873025 (2002m:46083)}

\bibitem[\textsc{Tay70}]{Tay1970}
\textsc{J.~L. Taylor}, \emph{The analytic-functional calculus for several
  commuting operators}, Acta Math. \textbf{125} (1970), 1--38. \MR{0271741 (42
  \#6622)}

\bibitem[\textsc{Tit58}]{Tit:TF}
\textsc{E.~C. Titchmarsh}, \emph{The theory of functions}, Oxford University
  Press, Oxford, 1958, Reprint of the second (1939) edition. \MR{3155290}

\bibitem[\textsc{Tr{\`e}06}]{Tre:TVS}
\textsc{F.~Tr{\`e}ves}, \emph{Topological vector spaces, distributions and
  kernels}, Dover Publications Inc., Mineola, NY, 2006, Unabridged
  republication of the 1967 original. \MR{2296978 (2007k:46002)}

\end{thebibliography}
\bibliographystyle{amsalpha-lmp}

\end{document}